\documentclass[pagesize ,11pt ,a4paper]{scrartcl}

\usepackage[ngerman, english]{babel}

\usepackage{
	amsmath, amsthm, amssymb, amsxtra, braket, tensor, mathrsfs, 		% Math typesetting
	accents, bbm, empheq, mathtools, mathdots, faktor,
	pstricks, graphicx, wrapfig, subcaption, bezier, capt-of,float,		% Figures and graphics formatting
	tikz, pgfplots, tkz-euclide, tikz-cd, 
	paralist, enumitem, mparhack, multicol, url, ae, wasysym, 			% Text and list formatting
	textcomp, bookmark, dsfont
	}

\usetikzlibrary{arrows, arrows.new, matrix, positioning}

\usepackage[left=1.15in, right=1.15in, top=1.2in, bottom=1in, headsep=.25in]{geometry}

\usepackage[bottom]{footmisc}

\allowdisplaybreaks

\usepackage[section]{placeins}

\setkomafont{section}{\normalsize}
\setkomafont{subsection}{\normalsize}

\usepackage[backend=bibtex, style=numeric]{biblatex}

\AtBeginBibliography{%
}

\DeclareNameAlias{default}{last-first}

\renewenvironment{abstract}
{\small\begin{quote}\noindent \par{\textsc \abstractname.}}
{\noindent\end{quote}}

\numberwithin{equation}{section}

\usepackage[title]{appendix}

\usepackage[T1]{fontenc}

\usepackage[euler-digits]{eulervm}

\usepackage[pstricks,squaren,cdot,thickqspace]{SIunits}

\usepackage{hyperref}
\usepackage[hyphenbreaks]{breakurl}

\usepackage[mathscr]{euscript}

\DeclareMathAlphabet{\mathpzc}{OT1}{pzc}{m}{it}

\newcommand{\N}{\mathbbm{N}}				% Natural numbers
\newcommand{\Z}{\mathbbm{Z}}				% Integers
\newcommand{\R}{\mathbbm{R}}				% Real numbers
\newcommand{\C}{\mathbbm{C}}				% Complex numbers
\newcommand{\CP}[1]{\C\mathrm{P}^{#1}}		% Complex projective space
\newcommand{\dv}{\mathrm{d}}				% Exterior derivative
\newcommand{\idmat}{\mathds{1}}				% Identity matrix
\newcommand{\mat}[2]{\mathrm{Mat}(#1,#2)}	% Set of matrices

\newcommand{\GL}[2]{\mathrm{GL}(#1,#2)}		% General linear group
\newcommand{\U}[1]{\mathrm{U}(#1)}			% Unitary group
\newcommand{\SU}[1]{\mathrm{SU}(#1)}		% Special unitary group
\newcommand{\gl}[2]{\mathfrak{gl}(#1,#2)}	% Lie algebra of general linear group
\let\u\relax								% Lie algebra of unitary group
\newcommand{\u}[1]{\mathfrak{u}(#1)}		
\newcommand{\su}[1]{\mathfrak{su}(#1)}		% Lie algebra of special unitary group
\DeclareMathOperator{\tr}{Tr}						% Trace
\DeclareMathOperator{\diag}{diag}					% Diagonal
\DeclareMathOperator{\res}{res}						% Residue
\DeclareMathOperator{\adj}{adj}						% Adjugate matrix

\newcommand{\rat}[3]{\mathrm{Rat}_{#1}(#2,#3)}			% Rational maps
\newcommand{\brat}[3]{\mathrm{Rat}^*_{#1}(#2,#3)}		% Based rational maps

\DeclarePairedDelimiter\floor{\lfloor}{\rfloor}

\newcommand{\setc}[2]{\left\lbrace#1\vphantom{#2}\right.\left|\vphantom{#1}\vphantom{#2}\right.\left.#2\vphantom{#1}\right\rbrace}

\newcommand{\abs}[1]{\left\lvert#1\right\rvert}

\newcommand{\fullstop}{\;\text.}
\newcommand{\comma}{\;\text,}

\theoremstyle{plain}
\newtheorem{theorem}{Theorem}[section]

\theoremstyle{plain}
\newtheorem*{theorem*}{Theorem}

\theoremstyle{plain}

\theoremstyle{plain}
\newtheorem{corollary}[theorem]{Corollary}

\theoremstyle{plain}
\newtheorem{proposition}[theorem]{Proposition}

\theoremstyle{plain}
\newtheorem{lemma}[theorem]{Lemma}

\theoremstyle{plain}
\newtheorem*{lemma*}{Lemma}

\theoremstyle{definition}
\newtheorem{definition}[theorem]{Definition}

\theoremstyle{definition}
\newtheorem*{definition*}{Definition}

\theoremstyle{remark}
\newtheorem{remark}[theorem]{Remark}

\theoremstyle{remark}
\newtheorem*{remark*}{Remark}

\theoremstyle{remark}

\theoremstyle{remark}
\newtheorem*{notation*}{Notation}

\newcounter{tmpa}
\newcounter{tmpb}
\newcounter{tmpc}

\begin{document}

\title{\normalsize\MakeUppercase{\bfseries 
Nahm's Equations and Rational Maps from $\CP{1}$ to $\CP{\lowercase{n}}$}}
\author{\small\textsc{
Max Schult}}
\date{\footnotesize\textsc\today}
\maketitle
\vspace{-1cm}

% --- Main content ---

\begin{abstract}
We consider Nahm's equations on a bounded open interval with first order poles at the ends. By imposing further boundary conditions, extended moduli spaces are identified with spaces of rational maps from $\CP{1}$ to $\CP{n}$. Having the residues at one end define sums of irreducible representations of $\su{2}$, the dimensions of those summands correspond to holomorphic charge on the rational map side. Technical difficulties arise due to half powers in the boundary conditions. Further, a symplectomorphic identification with moduli spaces corresponding to spaces of rational maps into complete flag varieties is given.
\end{abstract}
\section{Introduction}

Building on work of Nahm, Hitchin \cite{hitchin1983} described the moduli space of $\SU{2}$-monopoles on $\R^3$ in terms of solutions to Nahm's equations on an open interval with first order poles at both ends and residues defining irreducible representations of $\su{2}$. In his paper on the classification of monopoles \cite{donaldson1984}, Donaldson used this to identify the moduli space of framed $\SU{2}$-monopoles of charge $k$ with the set of based degree $k$ rational maps from $\CP{1}$ to $\CP{1}$. Hurtubise extended this result to $\SU{N}$-monopoles with maximal symmetry breaking by establishing a correspondence to rational maps into complete flag varieties \cite{hurtubise1989}. To our knowledge, rational maps into partial flag varieties have not been studied in this context. This leads us to dropping the irreducibility requirement of the residues at one end and consider for $n\geq 2$, $N=n+1$, a bounded open interval $\lambda=(\lambda_1,\lambda_N)\subset\R$ and $k\in\mathbb{N}^n$ with $m_1=m$ for $m_i=\sum_{j=i}^nk_j$, solutions $T\in C^\omega(\lambda,\u{m})\otimes\R^3$ to Nahm's equations
\begin{align}\label{nahms eq}
\dot{T}_i+[T_j,T_k]=0\comma
\end{align}
$(ijk)$ cyclic permutations of $(123)$, with the following boundary conditions:
\begin{enumerate}[label=(\roman*)]
\item $T_i$ is meromorphic near $\lambda_1$ with residue conjugate under $\U{m}$ to $\tau_i^m$, and
\item there exists $u\in\U{m}$ such that for $z<0$ of small modulus
\begin{align*}
uT_i(\lambda_N+z)u^{-1}=\frac{\diag\left(\tau_i^{k_n},\dots,\tau_i^{k_1}\right)}{z}+\left(\begin{matrix}z^\frac{k_{n}-k_n}{2}O(1)&\dots&z^\frac{k_1-k_n}{2}O(1)\\
\vdots&\ddots&\vdots\\
z^\frac{k_1-k_n}{2}O(1)&\dots&z^\frac{k_1-k_1}{2}O(1)\end{matrix}\right)\comma
\end{align*}
\end{enumerate}
where $\tau_j^\ell$ is the image of $\tau_j=-\frac{i}{2}\sigma_j$ under the standard $(\ell+1)$-dimensional representation of $\su{2}$ with $\sigma_j$ the $j^\text{th}$ Pauli matrix. $\U{m}$ acts on the set of all such solutions, so one can consider the quotient. As for monopoles, a natural bijective correspondence to spaces of rational maps arises when considering extended moduli spaces\footnote{Donaldson calls the moduli space of framed $\SU{2}$-monopoles the \emph{extended moduli space} \cite{donaldson1984}, inspiring our terminology here.}. Hence we consider tuples $(T,v)$ with $T$ as above and $v=(v_1,\dots,v_N)\in(\C^m)^*\times(\C^m)^n$ such that
\begin{enumerate}[label=(\roman*)]
\setcounter{enumi}{2}
\item $v_1^T$ lies in the $-i\frac{m-1}{2}$ eigenspace of $\res_{\lambda_1}(\alpha)^T$,
\item $v_{j+1}$ lies in the $i\frac{k_{j}-1}{2}$ eigenspace of $\res_{\lambda_N}(\alpha)$,
\item $\abs{v_{i}}=1$ and $\langle v_{i+1},v_{j+1}\rangle=0$,
\end{enumerate}
where $(\alpha,\beta)\in C^\omega(\lambda,\mathfrak{u}(m))\otimes(\R\times\C)$ corresponds to $T$ under a given isometric isomorphism $\R^3\cong\R\times\C$. $\res_{\lambda_N}$ here always refers to the $(t-\lambda_1)^{-1}$ coefficient, even if half powers appear. Induced by the standard action of $\U{m}$ on $\C^m$, one obtains an action of $\U{m}$ on the set $\mathcal{A}_k$ of all tuples $(T,v)$ as above. The quotient
\begin{align*}
\mathcal{M}_k=\faktor{\mathcal{A}_k}{\U{m}}
\end{align*}
then plays the role of the extended moduli space.

In \cite{murray1989}, rational maps into flag varieties are stratified, in addition to their degree, by their holomorphic charge. This turns out to correspond to the dimensions of the irreducible summands of the reducible representation of $\su{2}$, while the degree corresponds to the dimension of the irreducible representation. This is reflected in the main theorem of the present paper as follows.

\begin{theorem*}
Given an isomorphism $\R^3\cong\R\times\C$, compatible with the usual metrics, there is a natural bijective correspondence between the extended moduli space $\mathcal{M}_k$ and the quasi-affine algebraic variety $R_k$ of based rational maps from $\CP{1}$ to $\CP{n}$ of holomorphic charge $k$.
\end{theorem*}

Donaldson introduced the notion of \emph{based} rational maps from $\CP{1}$ to $\CP{1}$ as those with fixed value $0$ at $\infty$. Observing that this implies that a certain flag associated to the rational map is anti-standard at infinity leads to a natural generalization of the definition to rational maps from $\CP{1}$ to $\CP{n}$.

The proof of this theorem follows ideas of Donaldson in \cite{donaldson1984} and Hurtubise in \cite{hurtubise1989}. Appropriate adaptations of Nahm complexes and real Nahm complexes together with corresponding equivalence relations are introduced. The appearing half powers give rise to an additional difficulty when showing that every Nahm complex is equivalent to a real Nahm complex. Namely, solving the real equation up to order zero involves the construction of additional gauge transformations, yielding different approximate solutions.

Recent work of Charbonneau and Nagy \cite{charbonneaunagy2022} shows that for $\lambda_1<\lambda_N=-n\lambda_1$, solutions to Nahm's equations as considered here produce $\SU{N}$-monopoles with non-maximal symmetry breaking, as such solutions $T$ define Nahm data $\mathcal{N}=(\mathcal{F},\mathcal{T})$ with framing $\mathcal{F}=(\C^m,\C^m)$ and $\mathcal{T}=(\dv,T)$ of symmetry breaking type $\mathbb{T}=(\underline{\lambda},\underline{r},\underline{k})$ with $\underline{\lambda}=(\lambda_1,\lambda_N)$, $\underline{r}=(1,n)$ and $\underline{k}=(m,-k)$. The main theorem of the present paper then identifies the extended moduli space over the moduli space $\mathcal{M}^\mathbb{T}$ of Nahm data with symmetry breaking type $\mathbb{T}$ with the set of based rational maps from $\CP{1}$ to $\CP{n}$.

A similar identification of moduli spaces of solutions to Nahm's equations with spaces of rational maps from $\CP{1}$ to complete flag varieties has been given by Hurtubise in \cite{hurtubise1989} in the context of classifying monopoles. In the last two sections of the present paper, a symplectomorphism between certain such moduli spaces and $\mathcal{M}_k$ is given. Identifying the corresponding spaces of rational maps yields then another proof of the main result.

\section{Nahm's equations and Nahm complexes}\label{nahms equation and nahm complexes}

As in \cite{donaldson1984}, view Nahm's equations (\ref{nahms eq}) as the anti-self-duality equations $\dv A+A\wedge A=-{*}(\dv A+A\wedge A)$ for the connection $\dv+A$ with $A=\sum_{i=1}^3T_i\dv p_i$, where $(s,p_1,p_2,p_3)$ are standard coordinates on $\lambda\times\R^3$. $u\in C^\infty(\lambda,\U{m})$ acts on $\dv+A$ as
\begin{align*}
u.(\dv+A)=\dv-\dot uu^{-1}\dv s+uAu^{-1}\comma
\end{align*}
which leads to introducing a skew-adjoint $\dv s$-component $T_0$. The anti-self-duality equations then yield equations
\begin{align}\label{nahms eq with t0}
\dot T_i+[T_0,T_i]+[T_j,T_k]=0\comma
\end{align}
$(ijk)$ cyclic permutations of $(123)$, which are invariant under the action
\begin{align*}
u.(T_0,T_1,T_2,T_3)=(uT_0u^{-1}-\dot uu^{-1},uT_1u^{-1},uT_2u^{-1},uT_3u^{-1})\fullstop
\end{align*}
Changing to a gauge with $T_0=0$ then gives back Nahm's equations in the form (\ref{nahms eq}).

Using the isomorphism $\R^3\cong\R\times\C$ from the assumptions of the main theorem, one may introduce complex coordinates, e.g. $s+ip_1$, $p_2+ip_3$, and set, correspondingly,
\begin{align}\label{from t to a b}
\alpha=\tfrac{1}{2}(T_0+iT_1)\comma&&\beta=\tfrac{1}{2}(T_2+iT_3)\fullstop
\end{align}
The equations (\ref{nahms eq with t0}) then take the form
\begin{align}
\dot\beta+2[\alpha,\beta]&=0\comma&&\text{(``complex equation")}\label{complex eq}\\
F(\alpha,\beta)=(\alpha+\alpha^*)\dot{\hphantom{)}}+2([\alpha,\alpha^*]+[\beta,\beta^*])&=0\fullstop&&\text{(``real equation")}\label{real eq}
\end{align}
Define an action of $g\in C^\infty(\lambda,\GL{m}{\C})$ via
\begin{align*}
g.(\alpha,\beta)=(g\alpha g^{-1}-\tfrac{1}{2}\dot gg^{-1},g\beta g^{-1})
\end{align*}
and observe that this leaves the complex equation invariant, while the real equation is only invariant under unitary gauge transformations. In view of the extended moduli space $\mathcal{M}_k$ and the boundary conditions (i) to (v) in the introduction, this leads to the following definitions, analogous to those in \cite{hurtubise1989}. In preparation, denote
\begin{align*}
x(\ell)&=\diag\left(-\frac{\ell-1}{4},-\frac{\ell-3}{4},\dots,\frac{\ell-1}{4}\right)\comma&y(\ell)&=\left(\begin{matrix}0&0\\\idmat_{\ell-1}&0\end{matrix}\right)\comma
%a_1&=\diag\left(-\tfrac{k_1-1}{4},\dots,\tfrac{k_1-1}{4},-\tfrac{k_2-1}{4},\dots,\tfrac{k_2-1}{4}\right)\comma&a_2&=\diag\left(\tfrac{m-1}{4},\dots,-\tfrac{m-1}{4}\right)\comma\\
%b_1&=\left(\begin{matrix}0&0\\\idmat_{k_1-1}&0\end{matrix}\right)\oplus\left(\begin{matrix}0&0\\\idmat_{k_2-1}&0\end{matrix}\right)
%=\left(\small\begin{matrix}0\\1&0\\&\ddots&\ddots\\&&1&0\\&&&0&0\\&&&&1&0\\&&&&&\ddots&\ddots\\&&&&&&1&0\end{matrix}\right)
%\comma&b_2&=\left(\begin{matrix}0&\idmat_{m-1}\\0&0\end{matrix}\right)
%=\left(\small\begin{matrix}0&1\\&0&\ddots\\&&\ddots&1\\&&&0\end{matrix}\right)
%\fullstop
\end{align*}
and set $x_1=-x(m)$, $y_1=y(m)^T$ and
\begin{align*}
x_N=\diag\left(x(k_n),\dots,x(k_1)\right)\comma&&y_N=\diag\left(y(k_n),\dots,y(k_1)\right)\fullstop
\end{align*}
For ease of notation when referring to matrix entries of corresponding block matrices define
\begin{align*}
m_i=\sum_{j=i}^nk_j\fullstop
\end{align*}

\begin{definition}\label{def k nahm complex}
A tuple $((\alpha,\beta),v)$ is called \emph{$k$-Nahm complex} if $(\alpha,\beta)\in C^\infty(\lambda,\gl{m}{\C})$ is a solution of (\ref{complex eq}) satisfying the following boundary conditions:
\begin{enumerate}[label=(\roman*)]
\item $\alpha$ and $\beta$ are meromorphic near $\lambda_1$ with residues conjugate under $\GL{m}{\C}$ to $x_1$ and $y_1$,
\item there exists $g\in\GL{m}{\C}$ such that for $z<0$ of small modulus
\begin{align*}
g\alpha(\lambda_N+z)g^{-1}&=\frac{x_N}{z}+\left(\begin{matrix}z^\frac{k_{n}-k_n}{2}O(1)&\dots&z^\frac{k_1-k_n}{2}O(1)\\
\vdots&\ddots&\vdots\\
z^\frac{k_1-k_n}{2}O(1)&\dots&z^\frac{k_1-k_1}{2}O(1)\end{matrix}\right)\comma\\
g\beta(\lambda_N+z)g^{-1}&=\frac{y_N}{z}+\left(\begin{matrix}z^\frac{k_{n}-k_n}{2}O(1)&\dots&z^\frac{k_1-k_n}{2}O(1)\\
\vdots&\ddots&\vdots\\
z^\frac{k_1-k_n}{2}O(1)&\dots&z^\frac{k_1-k_1}{2}O(1)\end{matrix}\right)\comma
\end{align*}
\end{enumerate}
and $v=(v_1,\dots,v_N)\in(\C^m)^*\times(\C^m)^n$ such that
\begin{enumerate}[label=(\roman*)]
\setcounter{enumi}{2}
\item $v_1^T$ is an element of the $\frac{m-1}{4}$-eigenspace of $\res_{\lambda_1}(\alpha)^T$ and is cyclic for $\res_{\lambda_1}(\beta)^T$, and
\item $v_{i+1}$ is an element of the $-\frac{k_i-1}{4}$-eigenspace of $\res_{\lambda_N}(\alpha)$ with
\begin{align*}
\det\left(\begin{matrix}v_{N}&\dots&\res_{\lambda_N}(\beta)^{k_n-1}v_{N}&\dots&v_{2}&\dots&\res_{\lambda_N}(\beta)^{k_1-1}v_{2}\end{matrix}\right)\neq0\comma
\end{align*}
and $\res_{\lambda_N}(\beta)^{k_i}v_{i+1}=0$.
\end{enumerate}
\end{definition}

\begin{definition}
Denote by $\mathcal{G}$ the group of smooth $\GL{m}{\C}$-valued maps on $\lambda$ that extend continuously and invertibly to the boundary. Two $k$-Nahm complexes $((\alpha,\beta),v)$ and $((\alpha',\beta'),v')$ are defined to be equivalent iff there exists $g\in\mathcal{G}$ such that
\begin{align*}
(\alpha',\beta')=g.(\alpha,\beta)\comma&&v'=g.v=(v_1g^{-1},gv_2,\dots,gv_N)\fullstop
\end{align*}
Denote the arising quotient by $\mathcal{N}_k$.
\end{definition}

\begin{definition}\label{def real k nahm complex}
A \emph{real $k$-Nahm complex} is a $k$-Nahm complex $((\alpha,\beta),v)$ such that $(\alpha,\beta)$ solves (\ref{real eq}) and satisfies the boundary conditions:
\begin{enumerate}[label=(\roman*)]
\item $\res_{\lambda_1}(\alpha)$ and $\res_{\lambda_1}(\beta)$ are conjugate under $\U{m}$ to $i\tau_1^m$ and $\tau_2^m+i\tau_3^m$,
\item there exists $u\in\U{m}$ such that for $z<0$ of small modulus
\begin{align*}
&u\alpha(\lambda_N+z)u^{-1}=\frac{\diag\left(i\tau_1^{k_1},\dots,i\tau_1^{k_n}\right)}{2z}+\left(\begin{matrix}z^\frac{k_{n}-k_n}{2}O(1)&\dots&z^\frac{k_1-k_n}{2}O(1)\\
\vdots&\ddots&\vdots\\
z^\frac{k_1-k_n}{2}O(1)&\dots&z^\frac{k_1-k_1}{2}O(1)\end{matrix}\right)\comma\\
&u\beta(\lambda_N+z)u^{-1}\\
&\quad=\frac{\diag\left(\tau_2^{k_1}+i\tau_3^{k_1},\dots,\tau_2^{k_n}+i\tau_3^{k_n}\right)}{2z}+\left(\begin{matrix}z^\frac{k_{n}-k_n}{2}O(1)&\dots&z^\frac{k_1-k_n}{2}O(1)\\
\vdots&\ddots&\vdots\\
z^\frac{k_1-k_n}{2}O(1)&\dots&z^\frac{k_1-k_1}{2}O(1)\end{matrix}\right)\comma
\end{align*}
where $\tau^{k_i}_j$ have been defined in the introduction,
\end{enumerate}
and with $v=(v_1,\dots,v_N)$ satisfying
\begin{enumerate}[label=(\roman*)]
\setcounter{enumi}{2}
\item $\abs{v_{i}}=1$ and $\langle v_{i+1},v_{j+1}\rangle=0$.
\end{enumerate}
\end{definition}

\begin{definition}
Denote by $\mathcal{G}^\R$ the subgroup of $\mathcal{G}$ of $\U{m}$-valued maps. Two real $k$-Nahm complexes are defined to be equivalent iff they differ by an element of $\mathcal{G}^\R$. Denote the arising quotient by $\mathcal{N}^\R_k$.
\end{definition}

\begin{proposition}\label{real nahm complexes and nahms eq}
(\ref{from t to a b}) induces a bijective correspondence between $\mathcal{M}_k$ and $\mathcal{N}_k^\R$.
\begin{proof}
The map sending $(T,v)$ to $((\alpha,\beta),v)$ according to (\ref{from t to a b}) induces a well-defined injective map on equivalence classes. For the other direction, invert the relation (\ref{from t to a b}) and then change to a gauge with $T_0=0$. Note that if $T'$ corresponds to $(\alpha,\beta)$ via (\ref{from t to a b}), changing to a gauge with $T_0=0$ is equivalent to solving $\dot g=T_0'g$. From the boundary conditions for $\alpha$ it follows that conjugation by $g$ preserves the given boundary conditions.
\end{proof}
\end{proposition}

The following result shows that every equivalence class of $k$-Nahm complexes has a representative of a simplified form. This is useful for computations, which will be exploited in section \ref{solving the real equation}, and gives a recipe for constructing $k$-Nahm complexes in section \ref{the classification of k nahm complexes}.

\begin{proposition}[``Normal form"]\label{normal form}
Let $((\alpha,\beta),v)$ be a $k$-Nahm complex.
\begin{enumerate}[label=(\alph*)]
\item $(\alpha,\beta)$ is locally equivalent (i.e. by an element of $\mathcal{G}$) to $(0,\beta_0)$, where $\beta_0$ is constant.
\item The given $k$-Nahm complex is equivalent to a $k$-Nahm complex $((\alpha_{\mathrm{st}},\beta_{\mathrm{st}}),v_{\mathrm{st}})$ with
\begin{align*}
(g_\mathrm{st}.\alpha_{\mathrm{st}})(\lambda_1+z)&=-\frac{x_1}{z}\comma&
\beta_{\mathrm{st}}(\lambda_1+z)&=\left(\small\begin{matrix}0&&&&-z^{m-1}q_1\\z^{-1}&0&\\&\ddots&\ddots&&\vdots\\&&z^{-1}&0&\\&&&z^{-1}&-q_m\end{matrix}\right)^T\comma
\end{align*}
for $z>0$ small with $q\in\C^m$ and
\begin{align*}
g_\mathrm{st}(\lambda_1+z)=\left(\begin{matrix}
z^{m-1}q_2&\dots&zq_m&1\\
\vdots&\iddots&1\\
zq_m&\iddots\\
1
\end{matrix}\right)\comma
\end{align*}
further
\begin{align*}
\alpha_{\mathrm{st}}(\lambda_N+z)&=\frac{x_N}{z}
%\beta_{\mathrm{st}}(\lambda_1+z)&=\left(\small\begin{matrix}0&&&&-z^{k_1-1}%C_{11}&&&&-z^{\frac{k_1+k_2}{2}-1}C_{12}\\
%z^{-1}&0\\
%&\ddots&\ddots&&\vdots&&&&\vdots\\
%&&z^{-1}&0\\
%&&&z^{-1}&-C_{k_11}&&&&-z^\frac{k_2-k_1}{2}C_{k_12}\\
%&&&&&0\\
%&&&&\vdots&z^{-1}&0&&\vdots\\
%&&&&&&\ddots&\ddots\\
%&&&&-z^\frac{k_1-k_2}{2}C_{m1}&&&z^{-1}&-C_{m2}\end{matrix}\right)\comma
\end{align*}
and $\beta_\mathrm{st}(\lambda_N+z)$ has $(i,i)$-block
\begin{align*}
\left(\begin{matrix}0&&&&z^{k_{N-i}-1}C_{(m_{N+1-i}+1)i}\\z^{-1}&0&\\&\ddots&\ddots&&\vdots\\&&z^{-1}&0&\\&&&z^{-1}&C_{m_{N-i}i}\end{matrix}\right)\comma
\end{align*}
and $(i,j)$-block with $i\neq j$
\begin{align*}
\left(\begin{matrix}0&\dots&0&z^{\frac{k_{N-j}+k_{N-i}}{2}-1}C_{(m_{N+1-i}+1)j}\\
\vdots&\ddots&\vdots&\vdots\\
0&\dots&0&z^\frac{k_{N-j}-k_{N-i}}{2}C_{m_{N-i}j}\end{matrix}\right)=z^\frac{\abs{k_{N-i}-k_{N-j}}}{2}O(1)
\end{align*}
for $z<0$ of small modulus with $C\in\mat{m\times n}{\C}$ and $m_N=0$, and
\begin{align*}
v_{\mathrm{st}}=(e_1^T,e_{m_{n-1}+1},\dots,e_{m_0+1})\fullstop
\end{align*}
\end{enumerate}
\begin{proof}
The result in (a) is obtained by locally solving
\begin{align*}
\dot u=-2\alpha u\fullstop
\end{align*}
For a basis of solutions $(u_i)_{i=1}^m$ write $S^{-1}=(u_1,\dots,u_m)$ and observe that $S$ transforms $(\alpha,\beta)$ to $(0,\beta_0)$. For part (b) start by assuming that $\res_{\lambda_1}(\alpha)=-x_1$, $\res_{\lambda_1}(\beta)=y_1^T$, (ii) in the definition of a $k$-Nahm complex holds for $g=\idmat$ and $v=(e_m^T,{v_\mathrm{st}}_2,\dots,{v_\mathrm{st}}_N)$, as this can be achieved by a gauge transformation constant near the ends of $\lambda$. Let $u_{i}$ be the unique solutions from the following lemma.

\begingroup
\setcounter{tmpa}{\value{theorem}}
\setcounter{theorem}{0}
\renewcommand\thetheorem{\ref{parallel transport solutions}}
\begin{lemma}
Let $((\alpha,\beta),(e_m^T,{v_{\mathrm{st}}}_2,\dots,{v_\mathrm{st}}_N))$ be a $k$-Nahm complex with $\res_{\lambda_1}(\alpha)=-x_1$ and $\res_{\lambda_1}(\beta)=y_1^T$ and that satisfies (ii) in the definition of a $k$-Nahm complex for $g=\idmat$. Then there exist unique smooth solutions on $\lambda$ to the following problems:
\begingroup
\setcounter{tmpb}{\value{equation}}
\setcounter{equation}{0}
\setcounter{tmpc}{1}
\renewcommand{\theequation}{\Alph{tmpc}.\arabic{equation}}
\begin{align}
&\begin{dcases}\dot{u}_1=-2\alpha u_1\\\lim_{z\searrow0}z^{-\frac{m-1}{2}}u_1(\lambda_1+z)=e_1\end{dcases}\\
&\begin{dcases}\dot{u}_{i+1}=-2\alpha u_{i+1}\\\lim_{z\nearrow0}z^{-\frac{k_{N-i}-1}{2}}u_{i+1}(\lambda_N+z)={v_\mathrm{st}}_{i+1}\\
\forall j<N-i:\forall\ell\in\{m_{N-i-j}+1,\dots,m_{n-i-j}\}:\lim_{z\nearrow0}z^{-\frac{k_{N-j}-1}{2}}(u_{i+1})_\ell(\lambda_N+z)=0\end{dcases}
%&\begin{dcases}\dot u_{11}=-2\alpha u_{11}\\\lim_{z\searrow0}z^{-\frac{k_1-1}{2}}u_{11}(\lambda_1+z)={v_\mathrm{st}}_{11}\end{dcases}\comma\\
%&\begin{dcases}\dot u_{12}=-2\alpha u_{12}\\\lim_{z\searrow 0}z^{-\frac{k_2-1}{2}}u_{12}(\lambda_1+z)={v_\mathrm{st}}_{12}\\\forall i\in\{1,\dots,k_1\}:\lim_{z\searrow 0}z^{-\frac{k_1-1}{2}}(u_{12})_i(\lambda_1+z)=0\end{dcases}\comma\\
%&\begin{dcases}\dot u_2=2u_2\alpha\\\lim_{z\nearrow0}z^{-\frac{m-1}{2}}u_2(\lambda_2+z)={v_\mathrm{st}}_2\end{dcases}\fullstop
\end{align}
\endgroup
\setcounter{equation}{\thetmpb}
\end{lemma}
\endgroup
\setcounter{theorem}{\thetmpa}
\noindent
This may be viewed as parallel transporting the vectors $e_1$ and ${v_\mathrm{st}}_{i+1}$ via the connection corresponding to $\alpha$ from one end to the interior of the interval. This allows to relate $v_{i+1}$ and $v_1$ by transporting to the same point. This will also be crucial later when defining the corresponding rational map.

Now, take a smooth $\GL{m}{\C}$-valued function $\tilde S^{-1}$ on $\lambda$ such that near $\lambda_1$
\begin{align*}
\tilde S^{-1}=\left(\begin{matrix}u_1&\dots&\beta^{m-1}u_1\end{matrix}\right)
\end{align*}
and near $\lambda_N$
\begin{align*}
\tilde S^{-1}=\left(\begin{matrix}u_{N}&\dots&\beta^{k_n-1}u_{N}&\dots&u_{2}&\dots&\beta^{k_1-1}u_{2}\end{matrix}\right)\fullstop
\end{align*}
Then $\tilde S$ transforms $(\alpha,\beta)$ to $(0,\beta_0)$ near each end of the interval for some $\beta_0$ constant. Setting $S^{-1}(\lambda_i+z)=\tilde S^{-1}(\lambda_i+z)z^{(-1)^{\delta_{1i}}2x_i}$ then defines an element $S$ of $\mathcal{G}$ with $S(\lambda_i)=\idmat$ that takes $(\alpha,\beta)$ near the ends of $\lambda$ to the form
\begin{align*}
S.\alpha(\lambda_i+z)=(-1)^{\delta_{1i}}\frac{x_i}{z}\comma&&S.\beta(\lambda_1+z)=\left(\begin{matrix}
0&&&-z^{m-1}q_1\\
z^{-1}&\ddots&&\vdots\\
&\ddots&0&-zq_{m-1}\\
&&z^{-1}&-q_m
\end{matrix}\right)\comma
\end{align*}
and $\beta$ of the desired form near $\lambda_N$ up to ckecking the vanishing of the coefficients of negative powers on the off-diagonal blocks, as $\tilde S\beta\tilde S^{-1}$ is the matrix of $\beta$ with resprect to the basis $(u_1,\dots,\beta^{m-1}u_1)$ near $\lambda_1$ and with respect to the basis $(u_{2},...,\beta^{k_n-1}u_{n})$ near $\lambda_N$. Lemma \ref{conjugation to normal form} then implies, after showing the boundedness of the off-diagonal blocks of $\beta$ near $\lambda_N$, that $g_\mathrm{st}^{-1}S$, with $g_\mathrm{st}\equiv\idmat$ near $\lambda_N$, transforms $(\alpha,\beta)$ to the desired form.

Note that by having identical boundary conditions, $\tilde S^{-1}$ agrees near $\lambda_N$ with the fundamental solutions obtained in the proof of lemma \ref{parallel transport solutions}, so that $S(\lambda_N+z)=\idmat+z^{1/2}O(1)$. It then follows from (\ref{problem for u12}) that
\begin{align*}
\forall j<N-i:\forall\ell\in\{m_{N-i-j}+1,\dots,m_{n-i-j}\}:{}&S_{\ell(m_{N+1-i}+1)}^{-1}(\lambda_N+z)\\
&\quad=z^{-\frac{k_{N-i}-1}{2}}(u_{i+1})_\ell(\lambda_N+z)\\
&\quad=z^{\frac{k_{N-j}-k_{N-i}+1}{2}}O(1)\fullstop
\end{align*}
Hence,
\begin{align*}
S(\lambda_1+z)&=\idmat+\left(\begin{matrix}z^\frac{k_{n}-k_n+1}{2}O(1)&\dots&z^{1/2}O(1)\\
\vdots&\ddots&\vdots\\
z^\frac{k_1-k_n+1}{2}O(1)&\dots&z^\frac{k_1-k_1+1}{2}O(1)\end{matrix}\right)\comma\\
S^{-1}(\lambda_1+z)&=\idmat+\left(\begin{matrix}z^\frac{k_{n}-k_n+1}{2}O(1)&\dots&z^{1/2}O(1)\\
\vdots&\ddots&\vdots\\
z^\frac{k_1-k_n+1}{2}O(1)&\dots&z^\frac{k_1-k_1+1}{2}O(1)\end{matrix}\right)\comma
\end{align*}
so that the lower $(i,j)$-block of $\beta_\mathrm{st}(\lambda_N+z)$ vanishes of order $\frac{\abs{k_{N-i}-k_{N-j}}}{2}$ and, trivially, so does the upper $(j,i)$-block.
% An analogous argument for the case where $m$ is odd concludes the proof.
\end{proof}
\end{proposition}

\section{Solving the real equation}\label{solving the real equation}

It is shown in this section that every equivalence class of $k$-Nahm complexes contains a unique equivalence class of real $k$-Nahm complexes. For $\varepsilon>0$ small enough set $\lambda(\varepsilon)=(\lambda_1+\varepsilon,\lambda_N-\varepsilon)$ and fix a solution $(\alpha,\beta)$ to the complex equation on $\overline{\lambda(\varepsilon)}$. As observed in \cite{donaldson1984}, the equation $F(g.(\alpha,\beta))=0$ on $\lambda(\varepsilon)$ then is the Euler-Lagrange equation of the functional
\begin{align*}
\mathcal{L}:g\mapsto\int_{\lambda(\varepsilon)}\left(\abs{\alpha'+\alpha'^*}^2+2\abs{\beta'}^2\right)\dv s\comma
\end{align*}
$(\alpha',\beta')=g.(\alpha,\beta)$, on smooth $\GL{m}{\C}$-valued functions on $\lambda(\varepsilon)$. Observing that the integrand is invariant under $\U{m}$-valued functions leads to considering functions with values in $\mathrm{H}(m)=\GL{m}{\C}/\U{m}$, the space of positive hermitian matrices. For $g:\lambda(\varepsilon)\to\GL{m}{\C}$ set $h(g)=g^*g:\lambda(\varepsilon)\to\mathrm{H}(m)$.

Using that $(\alpha,\beta)$ can be gauged to $(0,\beta_0)$ on $\lambda(\varepsilon)$, Donaldson obtained the following existence and uniqueness result.

\begin{proposition}[{\cite[pp. 395-397]{donaldson1984}}]\label{ex and un on eps}
For any $h_+,h_-\in\mathrm{H}(m)$ there exists a continuous $g:\overline{\lambda(\varepsilon)}\to\GL{m}{\C}$, smooth on $\lambda(\varepsilon)$, with $h(g(\lambda_1+\varepsilon))=h_-$ and $h(g(\lambda_N-\varepsilon))=h_+$ such that $g.(0,\beta_0)$ solves the real equation over $\lambda(\varepsilon)$. Moreover, if $g_1$ and $g_2$ both satisfy these conditions over $\lambda(\varepsilon)$ (over $\lambda$), then $h(g_1)=h(g_2)$ over $\lambda(\varepsilon)$ (over $\lambda$), so the solution is unique up to unitary gauge transformations.
\end{proposition}

While the existence is deduced using the direct method of calculus of variations, the key ingredient for uniqueness is the following weak differential inequality for the eigenvalues of $h$. For $h\in\mathrm{H}(m)$ with eigenvalues $\lambda_i$ define $\Phi(h)=\log(\max_i(\lambda_i))$.

\begin{lemma}[{\cite[p. 396]{donaldson1984}}]
If $g$ is a smooth $\GL{m}{\C}$-valued function on some interval in $\lambda$, then
\begin{align*}
\frac{\dv^2}{\dv t^2}\Phi(h(g))\geq-2\left(\abs{F(\alpha,\beta)}-\abs{F(g.(\alpha,\beta))}\right)\comma\\
\frac{\dv^2}{\dv t^2}\Phi(h(g)^{-1})\geq-2\left(\abs{F(\alpha,\beta)}-\abs{F(g.(\alpha,\beta))}\right)\comma
\end{align*}
in the weak sense.
\end{lemma}

To find an analogue of the first part of proposition \ref{ex and un on eps} for the interval $\lambda$, one brings a given Nahm complex into a ``nice" form, in order to make more use of the above differential inequality. At this point, a difficulty arises in the case $k_1-k_2=1$ that is not present in Donaldson's or Hurtubise's work. The following lemma addresses that issue before formulating an analogue of the result found in \cite{donaldson1984}.

\begingroup
\setcounter{tmpa}{\value{theorem}}
\setcounter{theorem}{0}
\renewcommand\thetheorem{\ref{power -1/2}}
\begin{lemma}
If $((\alpha,\beta),v)$ is a $k$-Nahm complex, then there exists an equivalent $k$-Nahm complex $((\alpha',\beta'),v')$ such that the power $-\frac{1}{2}$-term in the expansion of $F(\alpha',\beta')$ near $\lambda_N$ vanishes.
\end{lemma}
\endgroup
\setcounter{theorem}{\thetmpa}

\begin{lemma}\label{nice nahm complex}
For every $k$-Nahm complex $((\alpha,\beta),v)$ there is an equivalent ``nice" $k$-Nahm complex $((\alpha',\beta'),v')$ such that
\begin{enumerate}[label=(\roman*), itemsep=0mm]
\item $F(\alpha',\beta')$ is bounded,
\item $\abs{\alpha'-\alpha'^*}$ is bounded,
\item $\abs{v_{i}}=1$ and $\langle v_{i+1},v_{j+1}\rangle=0$,
\item $\res_{\lambda_1}(T_i)$ are conjugate under $\U{m}$ to $\tau_i^m$ and $\res_{\lambda_N}(T_i)$ are conjugate under $\U{m}$ to $\diag(\tau_i^{k_n},\dots,\tau_i^{k_1})$, where $T_i$ correspond to $(\alpha',\beta')$ via (\ref{from t to a b}).
\end{enumerate}
\begin{proof}
By lemma \ref{power -1/2}, one may assume in the expansion of $F(\alpha,\beta)$ near $\lambda_N$ there is no power $-\frac{1}{2}$-term. Also assume condition (ii) to (iv) to hold already, as these can be achieved by a gauge transformation that is constant near the ends of the interval, which preserves the vanishing of the power $-\frac{1}{2}$-term. $\res_{\lambda_i}(F(\alpha,\beta))$ can then be gauged to zero as in \cite[p. 399]{donaldson1984}, where near one can set $g=\exp((\cdot)\chi)$ for $\chi=\diag(\chi_1,\dots,\chi_n)$ and find solutions for $\chi_i$ individually by the same argument.
\end{proof}
\end{lemma}

With this, one has

\begin{theorem}[{\cite[pp. 399-402]{donaldson1984}}]\label{almost real nahm complex}
Let $((\alpha,\beta),v)$ be a $k$-Nahm complex, nice in the sense of lemma \ref{nice nahm complex}. Then there exists $g\in C^0(\overline{\lambda},\GL{m}{\C})$, smooth on the interior, with $h(g(\lambda_i))=\idmat$ and $\dot g$ bounded in $\lambda$ such that
\begin{enumerate}[label=\arabic*.]
\item $(\alpha',\beta')=g.(\alpha,\beta)$ solves the real equation, with $\alpha'=\alpha'^*$ over $\lambda$, i.e. $T=(T_1,T_2,T_3)$ defined from $(\alpha',\beta')$ via (\ref{from t to a b}) is a solution to Nahm's equations.
\item There exists $g_1\in\U{m}$ with
\begin{align*}
\abs{u_1T_i(\lambda_2+z)g_1^{-1}-\frac{\tau_i^m}{z}}
\end{align*}
bounded in for $z>0$ small, and there exists $g_N\in\U{m}$ with
\begin{align*}
\abs{g_NT_i(\lambda_N+z)g_N^{-1}-\frac{\diag\left(\tau_i^{k_n},\dots,\tau_i^{k_1}\right)}{z}}
\end{align*}
bounded for $z<0$ of small modulus.
\end{enumerate}
\end{theorem}

To conclude the existence of an equivalence class of real $k$-Nahm complexes in each equivalence class of $k$-Nahm complexes it remains to show the appropriate analytic behavior near the ends of $\lambda$ for the solution obtained in the above theorem. Take now $(\alpha',\beta')$ as in theorem \ref{almost real nahm complex} with $u_N=\idmat$ in 2.. In view of the interpretation of Nahm's equations as a linear flow on the Jacobian of the corresponding spectral curve write
\begin{align*}
A(\zeta)=A_0+A_1\zeta+A_2\zeta^2=2\beta'+4\alpha'\zeta-2\beta'^*\zeta^2\fullstop
\end{align*}
Using this perspective, Hurtubise then shows the following.

\begin{lemma}[{\cite[pp. 623-624]{hurtubise1989}}]
There exists a polynomial $B(\zeta)=B_0+B_1\zeta+B_2\zeta^2$ with meromorphic coefficients and $B_0$ constant, and $f\in C^\infty(\lambda,\GL{m}{\C})$ such that
\begin{align*}
A(\zeta)=fB(\zeta)f^{-1}\comma&&A_1=-2\dot{f}f^{-1}\fullstop
\end{align*}
\end{lemma}

Note that in this situation $(\alpha',\beta')=f.(0,\beta_0)$ for $2\beta_0=B_0$. Since $g$ in the above theorem is unique up to constant unitary gauge transformations, proposition \ref{normal form}, lemma \ref{power -1/2}, lemma \ref{nice nahm complex} and theorem \ref{almost real nahm complex} imply that
\begin{align*}
f(\lambda_N+z)=(gk)(\lambda_N+z)z^{-2x_N}
\end{align*}
for $z<0$ of small modulus for some $k\in\mathcal{G}$ with
\begin{align*}
k(\lambda_N+z)=\idmat+\left(\begin{matrix}O(z)&z^\frac{3}{2}O(1)&\dots&z^\frac{3}{2}O(1)\\z^\frac{3}{2}O(1)&O(z)&\ddots&\vdots\\\vdots&\ddots&\ddots&z^\frac{3}{2}O(1)\\z^\frac{3}{2}O(1)&\dots&z^\frac{3}{2}O(1)&O(z)\end{matrix}\right)
\end{align*}
and $g\in\mathcal{G}$ with $\dot g$ bounded and $g(\lambda_N)=\idmat$. $B_1=-2f^{-1}\dot{f}$ then implies
\begin{align*}
M(\lambda_N+z)=\left((gk)^{-1}(gk)\dot{\hphantom{)}}\right)(\lambda_1+z)=z^{-2x_N}\left(-\frac{B_1}{2}-\frac{2x_N}{z}\right)z^{2x_N}
\end{align*}
for $z<0$ of small modulus. Hence, near $0$, $M(z)$ has diagonal blocks that are meromorphic in $z$ and off-diagonal blocks that are meromorphic in $z^{1-\frac{\rho}{2}}$ with $\rho\in\{0,1\}$ and $\rho$ and $m$ having the same parity. The boundedness of $(gk)\dot{\hphantom{)}}$ then implies
\begin{align*}
M(\lambda_1+z)=\left(\begin{matrix}O(z)&z^\frac{\rho}{2}O(1)&\dots&z^\frac{\rho}{2}O(1)\\z^\frac{\rho}{2}O(1)&O(z)&\ddots&\vdots\\\vdots&\ddots&\ddots&z^\frac{\rho}{2}O(1)\\z^\frac{\rho}{2}O(1)&\dots&z^\frac{\rho}{2}O(1)&O(z)\end{matrix}\right)\fullstop
\end{align*}
Solving for $gk$, one obtains a similar decomposition for $gk$ and thus also for $T_i$ just with meromorphic diagonal blocks. The following lemma together with condition 2. in theorem \ref{almost real nahm complex} then yields the appropriate meromorphic behavior of $T$ near $\lambda_N$.

\begingroup
\setcounter{tmpa}{\value{theorem}}
\setcounter{theorem}{0}
\renewcommand\thetheorem{\ref{representation theory gamma}}
\begin{lemma}
Let $T$ be a solution to Nahm's equations (\ref{nahms eq}) such that for $z<0$ of small modulus
\begin{align*}
T_i(\lambda_N+z)=\frac{\diag\left(\tau_i^{k_n},\dots,\tau_i^{k_1}\right)}{z}+\left(\begin{matrix}z^{\gamma_{11}}O(1)&\dots&z^{\gamma_{1n}}O(1)\\
\vdots&\ddots&\vdots\\
z^{\gamma_{n1}}O(1)&\dots&z^{\gamma_{nn}}O(1)\end{matrix}\right)
\end{align*}
for $\gamma_{ij}\geq0$. Then $2\gamma_{ij}\geq\abs{k_{N-i}-k_{N-j}}$ are integers.
\end{lemma}
\endgroup
\setcounter{theorem}{\thetmpa}

The meromorphicity of $T_i$ near $\lambda_1$ follows by a similar argument, so that condition 2. in theorem \ref{almost real nahm complex} yields the desired boundary conditions near $\lambda_1$, making $((\alpha',\beta'),v')$ with $\abs{v_i'}=1$ and $\langle v_{i+1}',v_{j+1}'\rangle=0$ a real $k$-Nahm complex.

There remains the question of uniqueness. Let now $((\alpha,\beta),v)$ be another real $k$-Nahm complex that is equivalent to $((\alpha',\beta'),v')$ as a $k$-Nahm complex via some $g\in\mathcal{G}$. The properties of the representation and the fact that $\abs{v_i'}=1$ imply $h(g(\lambda_i))=\idmat$. Proposition \ref{ex and un on eps} then implies $h(g_1)=h(g_2)$, giving the equivalence of $((\alpha,\beta),v)$ and $((\alpha',\beta'),v')$ as real $k$-Nahm complexes. One obtains

\begin{theorem}\label{unique real class}
Each equivalence class of $k$-Nahm complexes contains a unique equivalence class of real $k$-Nahm complexes.\qed
\end{theorem}

\section{The classification of $k$-Nahm complexes}\label{the classification of k nahm complexes}

In what follows, an adaptation of the approach by Donaldson in \cite[Section 3]{donaldson1984} is given in order to establish an analogous correspondence between $k$-Nahm complexes and certain rational maps.

\begin{proposition}\label{nahm complexes and matrices}
There is a bijective correspondence between
\begin{enumerate}[label=(\alph*), itemsep=0mm]
\item equivalence classes of $k$-Nahm complexes, i.e. the set $\mathcal{N}_k$, and
\item equivalence classes under $\GL{m}{\C}$ of pairs $(B,w)\in\GL{m}{\C}\times((\C^m)^*\times(C^m)^n)$ such that
\begin{enumerate}[label=(\roman*), itemsep=0mm]
%\item $B$ is an $m\times m$ matrix and $w=(w_{11},w_{12},w_2)$ with $w_{1i}\in\C^m$ column vectors and $w_2\in(\C^m)^*$ a row vector,
\item $w_1^T$ is a cyclic vector for $B^T$ and
\item $\det(M(k))\neq0$ and $\forall i>j:\forall l\in\{1,\dots,k_j-k_i\}:\det(M(k+l(e_i-e_j)))=0$, where for $\ell=(\ell_1,\dots,\ell_n)\in\N^n$ we define
\begin{align*}
M(\ell)=\left(\begin{matrix}w_{N}&\dots&B^{\ell_n-1}w_{N}&\dots&w_{2}&\dots&B^{\ell_1-1}w_{2}\end{matrix}\right)\fullstop
\end{align*}
\end{enumerate}
\end{enumerate}
\begin{proof}
Let $((\alpha,\beta),v)$ be a $k$-Nahm complex in normal form and let $g:\lambda\to\GL{m}{\C}$ be given near $\lambda_i$ by $g(\lambda_i+z)=z^{2x_i}$. Then define
\begin{align*}
(B,w)=\left((g.\beta)(\lambda_1),\left(e_1^T,(gu_2)(\lambda_1),\dots,(gu_N)(\lambda_1)\right)\right)\comma
\end{align*}
where $u_{i+1}$ are the solution in lemma \ref{parallel transport solutions} for $((\alpha,\beta),v)$. Note that
\begin{align*}
\frac{\dv}{\dv t}(gu_{i+1})=-2(g.\alpha)gu_{i+1}
\end{align*}
with $\alpha(\lambda_1+z)-x_1/z$ upper triangular and holomorphic in $z$ for $z>0$ small, so that $gu_{i+1}$ extends continuously to $\lambda_1$. From the construction it is immeadiate that $\det(M(k))\neq 0$ and that $w_1^T$ is a cyclic vector for $B^T$.

Note that in the basis $(w_{N},\dots,B^{k_n-1}w_{N},\dots,w_{2},\dots,B^{k_1-1}w_{2})$ the $(i,i)$-block of the matrix $B$ takes the form
\begin{align*}
\left(\begin{matrix}0&&&&C_{(m_{N+1-i}+1)i}\\1&0&\\&\ddots&\ddots&&\vdots\\&&1&0&\\&&&1&C_{m_{N-i}i}\end{matrix}\right)\comma
\end{align*}
the $(i,j)$-block for $i<j$ takes the form
\begin{align*}
\left(\begin{matrix}0&\dots&0&C_{(m_{N+1-i}+1)j}\\
\vdots&\ddots&\vdots&\vdots\\
0&\dots&0&C_{m_{N-i}j}\end{matrix}\right)\comma
\end{align*}
and the $(i,j)$-block for $i>j$ takes the form
\begin{align*}
\left(\begin{matrix}0&\dots&0&C_{(m_{N+1-i}+1)j}\\
\vdots&\ddots&\vdots&\vdots\\
0&\dots&0&C_{(m_{N+1-i}+k_{N-j})j}\\
0&\dots&0&0\\
\vdots&\ddots&\vdots&\vdots\\
0&\dots&0&0\end{matrix}\right)\comma
\end{align*}
with $C\in\mat{m\times n}{\C}$, as in proposition \ref{normal form}. This yields that for $i>j$ and $l\in\{1,\dots,k_j-k_i\}$ the matrix $M(k+l(e_i-e_j))$ is conjugate to one with $m_j^\text{th}$ row zero, leading to $\det(M(k+l(e_i-e_j))=0$, so that $(B,w)$ satisfies (i) and (ii) in the proposition.

For a $k$-Nahm complex $((\alpha',\beta'),v')$ in normal form equivalent to $((\alpha,\beta),v)$ via $g'\in\mathcal{G}$. Let $u_{i+1}'$ be the solutions in lemma \ref{parallel transport solutions} for $((\alpha',\beta'),v')$. Setting
\begin{align*}
(B',w')=\left((g.\beta')(\lambda_1),\left(e_1^T,gu'_2(\lambda_1),\dots,gu'_N(\lambda_1)\right)\right)\comma
\end{align*}
one immediately obtains that $B=B'$. Also $v'=v$ implies $g'(\lambda_i)=\idmat$, so that $g'u_{i+1}$ solve the problems in lemma \ref{parallel transport solutions} for $((\alpha',\beta'),v')$. Uniqueness then gives $u'_{i+1}=g'u_{i+1}$ and hence $w'=w$, so that the above yields a well-defined map on equivalence classes.

Let now $(B,w)$ be any tuple satisfying (i) and (ii) in the proposition. Then the matrix of $B$ in the basis $(w_1,\dots,w_1B^{m-1})$ is of the form
\begin{align*}
\left(\small\begin{matrix}0&&&&-q_1\\1&0&\\&\ddots&\ddots&&\vdots\\&&1&0&\\&&&1&-q_m\end{matrix}\right)^T\comma
\end{align*}
so that there exists $g_1\in\GL{m}{\C}$ that conjugates $B$ to that form and taking $w_1$ to $e_1^T$. Further, the matrix of $B$ in the basis $(w_{N},\dots,B^{k_n-1}w_{N},\dots,w_{2},\dots,B^{k_1-1}w_{2})$ has $(i,i)$-block
\begin{align*}
\left(\begin{matrix}0&&&&C_{(m_{N+1-i}+1)i}\\1&0&\\&\ddots&\ddots&&\vdots\\&&1&0&\\&&&1&C_{m_{N-i}i}\end{matrix}\right)\comma
\end{align*}
and $(i,j)$-block for $i\neq j$
\begin{align*}
\left(\begin{matrix}0&\dots&0&C_{(m_{N+1-i}+1)j}\\
\vdots&\ddots&\vdots&\vdots\\
0&\dots&0&C_{m_{N-i}j}\end{matrix}\right)\fullstop
\end{align*}
Let $g_N\in\GL{m}{\C}$ denote the corresponding change of basis. Observe that for $i>j$
\begin{align*}
0&=\det(M(k+(e_i-e_j)))\\
&=\abs{\begin{matrix}
\idmat_{m_{N-i}}&(C_{\ell i})_{\ell=1}^{m_{N-i}}\\
&(C_{\ell i})_{\ell=m_{N-i}+1}^{m_{N-j}-1}&\idmat_{m_{N-j}-m_{N-i}-1}\\
&C_{m_{N-j}i}&&0\\
&C_{(m_{N-j}+1)i}&&1\\
&(C_{\ell i})_{\ell=m_{N-j+2}}^m&&&\idmat_{m-m_{N-j}-1}
\end{matrix}}=\pm C_{m_{N-j}i}\comma
\end{align*}
and that if $(C_{(m_{N-j}-l)i},\dots,C_{m_{N-j}i})=0$ for $l\in\{0,\dots,k_1-k_2-1\}$, then
\setcounter{MaxMatrixCols}{20}
\begin{align*}
0&=\det(M(k+(l+1)(e_i-e_j)))=\pm C_{(m_{N-j}-(l+1))i}^{l+1}\comma
%&=\abs{\small\begin{matrix}
%1&&&&&&&&&&-C_{12}\\
%&\ddots&&&&&&&&&\vdots\\
%&&1&&&&&&&&-C_{(k_1-(n+1))2}&\ddots\\
%&&&0&&&&&&&-C_{(k_1-(n+1)+1)2}&\ddots&{*}\\
%&&&0&&&&&&&0&\ddots&\vdots\\
%&&&\vdots&&&&&&&\vdots&\ddots&{*}\\
%&&&0&&&&&&&0&\ddots&-C_{(k_1-(n+1)+1)2}\\
%&&&1&&&&&&&-C_{(k_1+1)2}&\ddots&0\\
%&&&&\ddots&&&&&&\vdots&\ddots&\vdots\\
%&&&&&1&&&&&-C_{(m-n)2}&\ddots&0\\
%&&&&&&1&&&&-C_{(m-n+1)2}&\ddots&*\\
%&&&&&&&&\ddots&&\vdots&&\vdots\\
%&&&&&&&&&1&-C_{m2}&&*
%\end{matrix}}\\
%&=\pm C_{(k_1-(n+1)+1)2}^{n+1}\comma
\end{align*}
as in the above basis and after moving the the block of columns $m_{N-i}+1$ to $m_{N-i}+l+1$ all the way to the right and then moving the block of rows $m_{N-i}+1$ to $m_{N-i}+l+1$ all the way down, the matrix $M(k+(l+1)(e_i-e_j))$ becomes of the form
\begin{align*}
\left(\begin{matrix}\idmat_{m-(l+1)}\\
&C_{(m_{N-j}-(l+1))i}&&{*}\\
&&\ddots&\\
&&&C_{(m_{N-j}-(l+1))i}
\end{matrix}\right)\fullstop
\end{align*}
This iteratively implies that $C_{(m_{N+1-j}+k_i+1)i}=\dots=C_{m_{N-j}i}=0$.

Let $g'\in\mathcal{G}$ be constant $g_qg_1$ near $\lambda_1$ and $g_N$ near $\lambda_N$, where
\begin{align*}
g_q=\left(\begin{matrix}
q_2&\dots&q_m&1\\
\vdots&\iddots&1\\
q_m&\iddots\\
1
\end{matrix}\right)\fullstop
\end{align*}
Then $g'.(0,B)$ solves the complex equation and is constant $(0,(g_1.B)^T)$ near $\lambda_1$ and $(0,g_N.B)$ near $\lambda_N$. A gauge transformation that equals $(\cdot-\lambda_1)^{-2x_1}g_\mathrm{st}(\cdot-\lambda_1)^{-2x_1}$ near $\lambda_1$ and $(\cdot-\lambda_1)^{-2x_N}$ near $\lambda_N$ then transforms this to a solution that together with $v_\mathrm{st}$ constitutes a $k$-Nahm complex in normal form. This construction yields an inverse for the above map on equivalence classes.
\end{proof}
\end{proposition}

Donaldson defined for a rational map from $\CP{1}$ to $\CP{1}$ what it means to be \emph{based}. This definition needs to be extended to include rational maps from $\CP{1}$ to $\CP{n}$ to suit the present situation before formulating an analogous correspondence to rational maps. Let $\zeta$ be the standard coordinate on $\CP{1}\setminus\{\infty\}$ with $\infty=[1,0]$ and consider degree $m$ rational maps $F\in\rat{m}{\CP{1}}{\CP{n}}$. Let $E$ be the trivial rank $N$ bundle over $\CP{1}$ and denote by $L(F)$ the line subbundle corresponding to $F$. Further, consider the exact sequence
\[
\begin{tikzcd}
0\arrow{r}{}&L(F)\arrow{r}{}&E\arrow{r}{}&\faktor{E}{L(F)}\arrow{r}{}&0
\end{tikzcd}\comma
\]
where, by the Birkhoff-Grothendieck theorem,
\begin{align*}
\faktor{E}{L(F)}\cong\bigoplus_{i=1}^n\mathcal{O}(\ell_n)
\end{align*}
for $\ell=(\ell_1,\dots,\ell_n\}\in\Z^n$ with $\ell_1\geq\dots\geq\ell_n$ and $\ell$ being called the \emph{holomorphic charge} of $F$, as in \cite[p. 663]{murray1989}. Denote by $\rat{\ell}{\CP{1}}{\CP{n}}$ the subset of $\rat{m}{\CP{1}}{\CP{n}}$ of all rational maps with holomorphic charge $\ell$. Note that $\deg(F)=m$ implies that $\deg(L(F))=-m$, so that $\deg(E/L(F))=m$ and
\begin{align*}
\sum_{i=1}^n\ell_i=m\fullstop
\end{align*}
From the long exact sequence in cohomology of the above sequence one obtains
\[
\begin{tikzcd}[row sep=tiny]
0\arrow{r}{}&H^0(\CP{1},L(F))\arrow{r}{}&H^0(\CP{1},E)\arrow{r}{}&H^0(\CP{1},E/L(F))\arrow{r}{}&\dots\\
\dots\arrow{r}{}&H^1(\CP{1},L(F))\arrow{r}{}&H^1(\CP{1},E)\arrow{r}{}&H^1(\CP{1},E/L(F))\arrow{r}{}&0
\end{tikzcd}\comma
\]
which by the Riemann-Roch theorem gives $H^1(\CP{1},E/L(F))=0$ and $h^0(\CP{1},E/L(F))=m+n$. This is only possible if all $\ell_i$ are non-negative. Considering now the dual short exact sequence
\[
\begin{tikzcd}
0\arrow{r}{}&\left(\faktor{E}{L(F)}\right)^*\arrow{r}{}&E^*\arrow{r}{}&L(F)^*\arrow{r}{}&0
\end{tikzcd}
\]
one observes that $(E/L(F))^*$ is a subbundle of the trivial bundle $E^*$. Fix a global ordered basis $(e_1,\dots,e_N)$ of $E^*$ and define the \emph{anti-standard flag} of bundles as
\begin{align*}
0\subset\langle e_N\rangle\subset\dots\subset\langle e_{N+1-i},\dots,e_N\rangle\subset\dots\subset\langle e_2,\dots,e_N\rangle\subset E\fullstop
\end{align*}

\begin{definition}\label{definition based}
$F\in\rat{m}{\CP{1}}{\CP{n}}$ is called \emph{based} if $F(\infty)=[1,0,\dots,0]$ and $(E/L(F))^*\cong\bigoplus_{i=1}^n\mathcal{O}(-\ell_i)$ admits subbundles $\bigoplus_{i=0}^j\mathcal{O}(-\ell_{n-i})$ such that
\begin{align*}
0\subset\mathcal{O}(-\ell_n)\subset\bigoplus_{i=0}^1\mathcal{O}(-\ell_{n-i})\subset\dots\subset\bigoplus_{i=0}^{n-2}\mathcal{O}(-\ell_{n-i})\subset\left(\faktor{E}{L(F)}\right)^*\subset E
\end{align*}
agrees with the anti-standard flag at $\infty$. Denote the subset of $\rat{\ell}{\CP{1}}{\CP{n}}$ of all based rational maps by $R_\ell$.
\end{definition}

With this generalization of the notion of based rational maps, we can formulate the following proposition, which establishes the desired correspondence to rational maps missing to obtain the main theorem.

\begin{proposition}\label{matrix to rational map}
The assignment of the rational map
\begin{align*}
(B,w)\mapsto F=[1,f_1,\dots,f_n]\in R_k\comma\quad f_i(z)=w_1(z\idmat-B)^{-1}w_{i+1}\in\C\cup\{\infty\}\comma
\end{align*}
induces a bijective correspondence between equivalence classes of pairs $(B,w)$ as in proposition \ref{nahm complexes and matrices} (b) and based rational maps of holomorphic charge $k$.
\end{proposition}

Before proving this proposition, it is worth reformulating for $F\in\rat{m}{\CP{1}}{\CP{n}}$ with $F(\infty)=\langle e_1\rangle|_\infty$ the condition $F\in R_k$ in a way better suited for computations. We start with the following characterization of negative degree line subbundles of $(E/L(F))^*$.

\begin{lemma}\label{line subbundles generalization}
For $F\in\rat{m}{\CP{1}}{\CP{n}}$ with $F(\infty)=\langle e_1\rangle|_\infty$ and $d>0$ there is a bijective correspondence between
\begin{enumerate}[label=(\roman*), itemsep=0mm]
\item degree $-d$ line subbundles of $(E/L(F))^*$ and
\item $\C^*$ equivalence classes of $(n+1)$-tuples $(s,t_1,\dots,t_n)$ of polynomials such that
\begin{enumerate}[label=(\alph*), itemsep=0mm]
\item $\deg(\gcd(s,t_1,\dots,t_n))=0$,
\item $\max(\deg(t_1),\dots,\deg(t_n))=d$, and
\item $sQ+\sum_{i=1}^nt_iP_i=0$,
\end{enumerate}
where $F$ is given on $\CP{1}\setminus\{\infty\}$ by $[(Q,P_1,\dots,P_n)\circ\zeta]$ for $(Q,P_1,\dots,P_n)$ an $n$-tuple of polynomials.
\end{enumerate}
\begin{proof}
A line subbundle $L$ of degree $-d$ of $(E/L(F))^*$ is determined by a non-vanishing meromorphic section with the only pole located at $\infty$ and of order $d$. Recalling the exact sequence
\[
\begin{tikzcd}
0\arrow{r}{}&\left(\faktor{E}{L(F)}\right)^*\arrow{r}{}&E^*\arrow{r}{}&L(F)^*\arrow{r}{}&0
\end{tikzcd}\comma
\]
such a meromorphic section is given on $\CP{1}\setminus\{\infty\}$ by an $N$-tuple $(s,t_1,\dots,t_n)$ of polynomials with $\deg(\gcd(s,t_1,\dots,t_n))=0$ and maximal degree $d$. Since $(E/L(F))^*$ is the kernel of the dual map to the inclusion $L(F)\hookrightarrow E$, $L$ being a subbundle of $(E/L(F))^*$ implies that
\begin{align*}
sQ+\sum_{i=1}^nt_iP_i=0\fullstop
\end{align*}
As $F(\infty)=\langle e_1\rangle|_\infty$ implies that $\deg(Q)>\deg(P_i)$, it follows that $\deg(s)<d$. Conversely, an $N$-tuple of polynomials as in (ii) defines a degree $-d$ line subbundle of $(E/L(F))^*$ by observing that multiplying $\zeta^{-d}$ gives a non-vanishing triple of polynomials in $\zeta^{-1}$.
\end{proof}
\end{lemma}

\begin{corollary}\label{cor line subbundles}
If $(E/L(F))^*\cong\bigoplus_{i=1}^n\mathcal{O}(-\ell_i)$ and $(s^i,t_1^i,\dots,t_n^i)$ corresponds to a summand $\mathcal{O}(-\ell_i)$, then for every degree $-d<0$ line subbundle $L$ of $(E/L(F))^*$ there exist polynomials $\lambda_i$ with $\deg(\lambda_i)=d-\ell_i$ such that $L$ corresponds to $\sum_{i=1}^n\lambda_i(s^i,t_1^i,\dots,t_n^i)$.\qed
\end{corollary}

For ease of notation we now consider for $F\in\rat{m}{\CP{1}}{\CP{n}}$ with $F(\infty)=\langle e_1\rangle|_\infty$ the following statement.
\begin{align}
  \tag{$A_{\ell}(F)$}
  \parbox{\dimexpr\linewidth-8em}{%
    \strut
    \emph{For all polynomials $s$ of degree at most $\max\setc{\ell_i}{i\in\{1,\dots,n\}}-2$ and for all non-zero $n$-tuples $(t_1,\dots ,t_n)$ of polynomials with $\deg(t_i)\leq\ell_i-1$ it holds that}
    $$
    sQ+\sum_{i=1}^nt_iP_i\neq 0\comma
    $$
    \emph{where $F$ is given on $\CP{1}\setminus\{\infty\}$ by $[(Q,P_1,\dots ,P_n)\circ\zeta]$ for $(Q,P_1,\dots ,P_n)$ an $(n+1)$-tuple of polynomials.}
    \strut
  } 
\end{align}

\noindent The aim is to characterize the condition for $F\in\rat{m}{\CP{1}}{\CP{n}}$ with $F(\infty)=\langle e_1\rangle|_\infty$ to be based in terms of such statements. More precisely do we want to show that this condition is precisely
\begin{align*}
A_k(F)\quad\land\quad\forall i>j:\forall l\in\{1,\dots,k_j-k_i\}:\neg A_{k+l(e_i-e_j)}(F)\comma
\end{align*}
which is of a similar form as
\begin{align*}
\det(M(k))\neq 0\quad\land\quad\forall i>j:\forall l\in\{1,\dots,k_j-k_i\}:\det(M(k+l(e_i-e_j)))=0\comma
\end{align*}
which is desireble in view of the correspondence we want to establish. This similarity will indeed be of relevance in the proof of proposition \ref{matrix to rational map}.

\begin{lemma}\label{characterization of rk}
Let $k=(k_1,\dots,k_n)$ with $k_1\geq\dots\geq k_n$ and let $F\in\rat{m}{\CP{1}}{\CP{n}}$ with $F(\infty)=\langle e_1\rangle|_\infty$. Then $F\in R_k$ iff
\begin{align*}
A_k(F)\quad\land\quad\forall i>j:\forall l\in\{1,\dots,k_i-k_j\}:\neg A_{k+l(e_j-e_i)}(F)\fullstop
\end{align*}
\begin{proof}
Let $F\in R_k$. Take a non-vanishing $(n+1)$-tuple $(s,t_1,\dots,t_n)$ of polynomials with
\begin{align*}
sQ+\sum_{i=1}^nt_iP_i=0\fullstop
\end{align*}
Then, according to lemma \ref{line subbundles generalization} and corollary \ref{cor line subbundles}, one has
\begin{align*}
(s,t_1,\dots,t_n)=\sum_{i=1}^n\lambda_i(s^i,t_1^i,\dots,t_n^i)
\end{align*}
with $\deg(\lambda_i)=d-k_i$, where $d=\max\{\deg(t_1),\dots,\deg(t_n)\}$. Thus, $d=k_i$ for some $i\in\{1,\dots,n\}$ and, further, $\lambda_j=0$ for $j<i$. Hence it follows $\deg(t_i)=k_i$, so that $A_k(F)$ follows.

Choosing the $(s^i,t_1^i,\dots,t_n^i)$ such that $\deg(t^i_j)<k_j$ for $j<i$, one obtains for $i>j$ and $l\in\{1,\dots,k_i-k_j\}$ that $\deg(t^j_j)=k_j<k_j+l$ and $\deg(t^j_i)<k_j\leq k_i-l$, so that $\neg A_{k+l(e_j-e_i)}(F)$.

Assume now
\begin{align*}
A_k(F)\quad\land\quad\forall i>j:\forall l\in\{1,\dots,k_i-k_j\}:\neg A_{k+l(e_j-e_i)}(F)\fullstop
\end{align*}
Then from $\neg A_{k+(e_i-e_1)}(F)$ for $i>i_0$, where $k_1$ has multiplicity $i_0$, one obtains an $(n+1)$-tuple $(s^i,t_1^i,\dots,t_n^i)$ of not simultaniously vanishing polynomials with
\begin{align*}
s^iQ+\sum_{j=1}^nt^i_jP_j=0\comma
\end{align*}
where $\deg(t^i_i)<k_i+1$, $\deg(t^i_1)<k_1-1$ and $\deg(t^i_j)<k_j$ for $j\notin\{1,i\}$. $A_k(F)$ then implies that $\deg(t^i_i)=k_i$. If $(\tilde s^i,\tilde t_1^i,\dots,\tilde t_n^i)$ corresponds to a summand $\mathcal{O}(-\ell_i)$ of $(E/L(F))^*\cong\bigoplus_{i=1}^n\mathcal{O}(-\ell_i)$, then
\begin{align*}
(s^i,t_1^i,\dots,t_n^i)=\sum_{i=1}^n\lambda^i_j(\tilde s^j,\tilde t_1^j,\dots,\tilde t_n^j)
\end{align*}
for polynomials $\lambda^i_j$ of degree $d_i-\ell_j$, where $d_i=\max\{\deg(t_1^i),\dots,\deg(t_n^i)\}$. But this implies that $\{k_{i_0+1},\dots,k_n\}\subset\{\ell_1,\dots,\ell_n\}$, and since $\deg(t^i_j)\leq k_j$ with equality only for $j=i$, $\{\ell_1,\dots,\ell_n\}\setminus\{k_{i_0+1},\dots,k_n\}$ has at most $i_0$ elements. If this set had in fact more than one element, there would be an index $i_1\in\{1,\dots,i_0\}$ such that $\ell_{i_1}<k_1$ and $\ell_{i_1}=\min\setc{\ell_i}{i\in\{1,\dots,i_0\}}$. But then
\begin{align*}
s^{i_1}Q+\sum_{j=1}^nt_j^{i_1}P_j=0
\end{align*}
would contradict $A_k(F)$, as $\max\{\deg(t^{i_1}_1),\dots,\deg(t^{i_1}_n)\}=\ell_{i_1}$ and $\deg(t^{i_1}_j)<k_j$ for $j>i_0$, possibly after subtracting appropriate polynomial multiples of the above $(s^i,t_1^i,\dots,t_n^i)$ for $i>i_0$. This implies that $\{\ell_1,\dots,\ell_n\}\setminus\{k_{i_0+1},\dots,k_n\}=\{k_1\}$, so that $F\in\rat{k}{\CP{1}}{\CP{n}}$. Linearly combining the tuples corresponding to line subbundles $\mathcal{O}(-k_1)$ of $(E/L(F))^*$ then yields the desired values at $\infty$, so that $F\in R_k$, as $(s^i,t_1^i,\dots,t_n^i)$ corresponds to $\mathcal{O}(-k_i)\subset(E/L(F))^*$ with $\mathcal{O}(-k_i)|_\infty=\langle e_{i+1}\rangle|_\infty$ for $i>i_0$.
\end{proof}
\end{lemma}

\begin{corollary}
$R_k$ is a quasi-affine algebraic variety of dimension $2\sum_{i=1}^nik_i$.
\begin{proof}
For $\ell\in\N^n$ with $\sum_{i=1}^n\ell_i=m$ and $\ell_\mathrm{max}=\max(\{\ell_1,\dots,\ell_n\})$, and polynomials $s,t_1,\dots,t_n$ with $\deg(s)<\ell_\mathrm{max}-1$ and $\deg(t_i)<\ell_i$ write $s(z)=\sum_{j=1}^{\ell_\mathrm{max}-2}s_jz^j$ and $t_i(z)=\sum_{j=1}^{\ell_i}t_{ij}z^j$. Then
\begin{align*}
&\left(sQ+\sum_{i=1}^nt_iP_i\right)(z)\\
&\quad=\left(\small\begin{matrix}s_{\ell_\mathrm{max}-2}\\\vdots\\s_0\\t_{1(\ell_1-1)}\\\vdots\\t_{10}\\\vdots\\t_{n(\ell_n-1)}\\\vdots\\t_{n0}\end{matrix}\right)^T\left(\small\begin{matrix}
1&q_{m-1}&\dots&q_0\\
&\ddots&\ddots&&\ddots\\
&&1&q_{m-1}&\dots&q_0\\
p_{1(m-1)}&\dots&p_{10}\\
&p_{1(m-1)}&\dots&p_{10}\\
&&\ddots&&\ddots\\
&&&p_{1(m-1)}&\dots&p_{10}\\
&&\vdots\\
&p_{n(m-1)}&\dots&p_{n0}\\
&&\ddots&&\ddots\\
&&&p_{n(m-1)}&\dots&p_{n0}
\end{matrix}\right)\left(\small\begin{matrix}z^{m+\ell_\mathrm{max}-2}\\\vdots\\z\\1\end{matrix}\right)\comma
\end{align*}
where $P_i(z)=\sum_{j=0}^{m-1}p_{ij}z^j$ and $Q(z)=\sum_{j=0}^mq_jz^j$ with $q_m=1$, and the respective blocks of the matrix are of hights $\ell_\mathrm{max}-1$ and $\ell_i$. Denoting this matrix by $\tilde M(\ell)$, $F=[Q,P_1,\dots,P_n]$ is in $R_k$ iff
\begin{align*}
\det(\tilde M(k))\neq 0\quad\land\quad\forall i>j:\forall l\in\{1,\dots,k_j-k_i\}:\det(\tilde M(k+l(e_i-e_j)))=0\comma
\end{align*}
so that $R_k$ can be identified with a subset of $\C^{N\cdot m}$ given by the vanishing of $\sum_{i<j}(k_i-k_j)$, which one inductively shows to be equal to $\sum_{i=1}^n(n+1-2i)k_i$, and the non-vanishing of one polynomial expression. Thus, as the polynomials are linearly independent, the dimension of $R_k$ is
\begin{align*}
(n+1)m-\sum_{i<j}(k_i-k_j)=(n+1)\sum_{i=1}^nk_i-\sum_{i=1}^n(n+1-2i)k_i=2\sum_{i=1}^nik_i\comma
\end{align*}
which concludes the proof.
\end{proof}
\end{corollary}

\begin{remark}
$R_{(1,\dots,1)}\subset\rat{n}{\CP{1}}{\CP{n}}$ is precisely the subset of full rational maps $F$ with $F(\infty)=\langle e_1\rangle|_\infty$.
\end{remark}

With this characterization of $R_k$ we can now prove proposition \ref{matrix to rational map}, the last missing piece in the proof of the main theorem.

\begin{proof}[Proof of proposition \ref{matrix to rational map}]
Let $(B,w)$ be a pair as in proposition \ref{nahm complexes and matrices} (b) such that
\begin{align*}
(w_1^T,\dots,(w_1B^{m-1})^T)=\idmat
\end{align*}
and let $F=[1,f_1,\dots,f_n]$ with $f_i(z)=e_1^T(z\idmat-B)^{-1}w_{i+1}$. Then $f_i=P_i/Q$ for polynomials $P_i$ and $Q$ given by
\begin{align*}
Q(z)=\det(z\idmat-B)\comma&&P_i(z)=e_1^T\adj(z\idmat-B)w_{i+1}\comma
\end{align*}
where $\deg(Q)=m$ and $\deg(P_i)<m$. Observe that for $\abs{z}$ large enough one has
\begin{align*}
f_i(z)=e_1^T(z\idmat-B)^{-1}w_{i+1}=\sum_{j=0}^\infty e_1^TB^jw_{i+1}z^{-j-1}
\end{align*}
Take now polynomials $t_i$ with $t_i(z)=\sum_{j=0}^{\ell_i-1}t_{ij}z^j$ for $\ell\in\N^n$ with $\sum_{i=1}^n\ell_i=m$ and compute for $\abs{z}$ large
\begin{align*}
&\left(\sum_{i=1}^nt_i\frac{P_i}{Q}\right)(z)=\sum_{i=0}^{\ell_1-1}\sum_{j=0}^{\infty}t_{1i}e_1^TB^jw_{2}z^{i-j-1}+\dots+\sum_{i=0}^{\ell_n-1}\sum_{j=0}^{\infty}t_{ni}e_1^TB^jw_{N}z^{i-j-1}\\
&\quad=O(1)+\sum_{j=0}^{\infty}\left(\sum_{i=0}^{\ell_1-1}t_{1i}e_1^TB^{i+j}w_{2}+\dots+\sum_{i=0}^{\ell_n-1}t_{ni}e_1^TB^{i+j}w_{N}\right)z^{-j-1}\fullstop
%&=O(1)+\sum_{j=0}^{m-1}\left(\sum_{i=0}^{\ell_1-1}t_{1i}(B^{i}w_{11})_{j+1}+\sum_{i=0}^{\ell_2-1}t_{2i}(B^{i}w_{12})_{j+1}\right)z^{-j-1}\\
%&=O(1)+\left(\begin{matrix}z^{-1}&\dots&z^{-m}\end{matrix}\right)M(\ell)\left(\begin{matrix}t_{11}\\\vdots\\t_{1(\ell_1-1)}\\t_{12}\\\vdots\\t_{2(\ell_2-1)}\end{matrix}\right)\comma
\end{align*}
If $Q$ and all $P_i$ had a common zero, there would exist a non-zero polynomial $\tilde Q$ with $\tilde Q(z)=\sum_{h=0}^{m-1}q_hz^h$ such that
\begin{align*}
&\left(\sum_{i=1}^nt_i\frac{P_i}{Q}\right)(z)\tilde Q(z)\\
&\quad=O(1)+\sum_{h=0}^{m-1}q_h\sum_{j=0}^{\infty}\left(\sum_{i=0}^{\ell_1-1}t_{1i}e_1^TB^{i+j}w_{2}+\dots+\sum_{i=0}^{\ell_n-1}t_{ni}e_1^TB^{i+j}w_{N}\right)z^{h-j-1}\\
&\quad=O(1)+\sum_{j=0}^\infty\sum_{h=0}^{m-1}q_h\left(\sum_{i=0}^{\ell_1-1}t_{1i}e_i^TB^{i+j+h}w_{2}+\dots+\sum_{i=0}^{\ell_n-1}t_{ni}e_1^TB^{i+j+h}w_{N}\right)z^{-j-1}\\
&\quad=O(1)+\left(\begin{matrix}z^{-1}&\dots&z^{-m}\end{matrix}\right)\left(\sum_{h=0}^{m-1}q_hB^h\right)M(\ell)\left(\begin{matrix}t_{11}\\\vdots\\t_{1(\ell_1-1)}\\\vdots\\t_{n1}\\\vdots\\t_{n(\ell_n-1)}\end{matrix}\right)\\
&\quad\phantom{{}={}}+\sum_{j=m}^\infty\sum_{h=0}^{m-1}q_h\left(\sum_{i=0}^{\ell_1-1}t_{1i}e_1^TB^{i+j+h}w_{2}+\dots+\sum_{i=0}^{\ell_n-1}t_{ni}e_1^TB^{i+j+h}w_{N}\right)z^{-j-1}
\end{align*}
is a polynomial in $z$. Since for $\ell=k$ one has $\det(M(k))\neq 0$ and since
\begin{align*}
e_1^T\left(\sum_{h=0}^{m-1}q_hB^h\right)=\sum_{h=0}^{m-1}q_he_{h+1}^T\neq 0\comma
\end{align*}
it is possible to choose $t_i$ such that this is not the case. This then shows that $Q$ and all $P_i$ have no common zeros, which yields $F\in\rat{m}{\CP{1}}{\CP{n}}$. In particular, for $\ell=k$
\begin{align*}
&\left(\sum_{i=1}^nt_i\frac{P_i}{Q}\right)(z)\\
&\quad=O(1)+\left(\begin{matrix}z^{-1}&\dots&z^{-m}\end{matrix}\right)M(k)\left(\begin{matrix}t_{11}\\\vdots\\t_{1(\ell_1-1)}\\\vdots\\t_{n1}\\\vdots\\t_{n(\ell_n-1)}\end{matrix}\right)\\
&\quad\phantom{{}={}}+\sum_{j=m}^\infty\sum_{h=0}^{m-1}q_h\left(\sum_{i=0}^{\ell_1-1}t_{1i}e_1^TB^{i+j+h}w_{2}+\dots+\sum_{i=0}^{\ell_n-1}t_{ni}e_1^TB^{i+j+h}w_{N}\right)z^{-j-1}
\end{align*}
is not a polynomial if not all $t_i$ are zero, so that for all non-zero $n$-tuples $(s,t_1,\dots,t_n)$ with $\deg(t_i)\leq k_i-1$ and $\deg(s)\leq k_1-2$ one has
\begin{align*}
sQ+\sum_{i=1}^nt_iP_i\neq 0\comma
\end{align*}
which yields $A_k(F)$. Further, compute
\begin{align*}
&\left(\sum_{i=1}^nt_i\frac{P_i}{Q}\right)(z)\\
&\quad=O(1)+\sum_{j=0}^{\infty}\left(\sum_{i=0}^{\ell_1-1}t_{1i}e_1^TB^{i+j}w_{2}+\dots+\sum_{i=0}^{\ell_n-1}t_{ni}e_1^TB^{i+j}w_{N}\right)z^{-j-1}\\
&\quad=O(1)+\sum_{h=0}^\infty\sum_{j=0}^{m-1}\left(\sum_{i=0}^{\ell_1-1}t_{1i}e_1^TB^{i+j+hm}w_{2}\right.\\
&\quad\phantom{=O(1)+\sum_{h=0}^\infty\sum_{j=0}^{m-1}\quad)}\left.+\dots+\sum_{i=0}^{\ell_n-1}t_{ni}e_1^TB^{i+j+hm}w_{N}\right)z^{-j-1-hm}\\
&\quad=O(1)+\sum_{h=0}^\infty\left(\begin{matrix}z^{-1}&\dots&z^{-m}\end{matrix}\right)\left(\frac{B}{z}\right)^{hm}M(\ell)\left(\begin{matrix}t_{11}\\\vdots\\t_{1(\ell_1-1)}\\\vdots\\t_{n1}\\\vdots\\t_{n(\ell_n-1)}\end{matrix}\right)\fullstop
%&\quad=\mathcal{O}(1)+\left(\begin{matrix}z^{-1}&\dots&z^{-m}\end{matrix}\right)\left(\idmat-\left(\frac{B}{z}\right)^{m}\right)^{-1}M(\ell)\left(\begin{matrix}t_{11}\\\vdots\\t_{1(\ell_1-1)}\\t_{21}\\\vdots\\t_{2(\ell_2-1)}\end{matrix}\right)\fullstop
\end{align*}
For $\ell=k+l(e_i-e_j)$ with $i>j$ and $l\in\{1,\dots,k_j-k_i\}$ the assumption gives $\det(M(\ell))=0$, so that choosing $t_i$ appropriately makes the right hand side a polynomial, denoted by $-s(z)$. One then obtains
\begin{align*}
sQ+\sum_{i=1}^nt_iP_i=0\comma
\end{align*}
which implies $\neg A_\ell(F)$. By lemma \ref{characterization of rk}, this shows $F\in R_k$.

To complete the proof, we give the inverse construction. Given $F\in R_k$, $F$ is of the form $[1,f_1,\dots,f_n]$ for $f_i=P_i/Q$ with $\deg(Q)=m$, $\deg(P_i)<m$, and $Q$ and all $P_i$ having no common zeros, as $F$ has degree $m$ and value $\langle e_1\rangle|_\infty$ at infinity. Let $V$ be the $m$-dimensional vector space $\C[t]/\langle Q(t)\rangle$. Define a bilinear form on $V$ via
\begin{align*}
(\pi,\sigma)=\sum_{\nu\in Z(Q)}\res_{\nu}\left(\frac{\pi\sigma}{Q(t)}\dv t\right)\comma
\end{align*}
where $Z(Q)$ is the set of zeros of $Q$. Note that using the primary decomposition and the Chinese remainder theorem, one obtains
\begin{align*}
V\cong\bigoplus_{\nu\in Z(Q)}V_\nu
\end{align*}
with
\begin{align*}
V_\nu=\setc{\sum_{j=0}^{\ell-1}\pi_j(t-\nu)^j}{(\pi_0,\dots,\pi_{\ell-1})\in\C^\ell}
\end{align*}
for $\nu\in Z(Q)$ of order $\ell$. Moreover, the summands of this decomposition are orthogonal with respect to $(\cdot,\cdot)$, and one has on each $V_\nu$ as above that
\begin{align*}
(\pi,\sigma)=\sum_{i=0}^{\ell-1}\pi_i\sigma_{\ell-i-1}\comma
\end{align*}
so that $(\cdot,\cdot)$ is non-degenerate on $V_\nu$ and hence on $V$. Define now $B:V\to V$ by $B(\pi)=t\pi$ and set $w_1=([1],\cdot)\in V$, where $[1]\in V$  corresponds to $1\in\C[t]$. Then, $w_1\circ B^j=([t^j],\cdot)$ for $j\in\{0,\dots,m-1\}$, so that $w_1^T$ is a cyclic vector for $B^T$. Further, set $w_{i+1}=[P_i(t)]\in V$ and observe that
\begin{align*}
\det\left(\begin{matrix}w_{2}&\dots&B^{\ell_1-1}w_{2}&\dots&w_{N}&\dots&B^{\ell_n-1}w_{N}\end{matrix}\right)=0
\end{align*}
iff there exist $t_{ij}\in\C$ with
\begin{align*}
\sum_{j=0}^{\ell_1-1}t_{1j}B^jw_{2}+\dots+\sum_{j=0}^{\ell_n-1}t_{nj}B^jw_{N}=0\in V\comma
\end{align*}
where $B^jw_{i+1}=[t^jP_i(t)]$. Setting $t_i(t)=\sum_{j=0}^{\ell_i-1}t_{ij}t^j$, this is equivalent to $A_{\ell}(F)$, which implies using $F\in R_k$ and lemma \ref{characterization of rk} that
\begin{align*}
A_k(F)\quad\land\quad\forall i<j:\forall l\in\{1,\dots,k_j-k_i\}:\neg A_{k+l(e_i-e_j)}(F)\comma
\end{align*}
so that $(B,w)$ in the identification $\C^m\cong V$ via the ordered basis $(1,t,\dots,t^{m-1})$ satisfies the conditions of proposition \ref{nahm complexes and matrices} (b). The rational map $G=[1,g_1,\dots,g_n]$ constructed from $(B,w)$ satisfies for $z\in\C$ with $z\notin Z(Q)$
\begin{align*}
g_i(z)&=w_1(z\idmat-B)^{-1}w_{i+1}=([1],[(z-t)^{-1}P_i(t)])\\
&=\sum_{\nu\in Z(Q)}\res_\nu\left(\frac{(z-t)^{-1}P_i(t)}{Q(t)}\dv t\right)=\sum_{\nu\in Z(Q)}\res_\nu\left(\frac{f_i(t)}{z-t}\dv t\right)\fullstop
\end{align*}
But
\begin{align*}
\sum_{\nu\in Z(Q)\cup\{z\}}\res_\nu\left(\frac{f_i(t)}{z-t}\dv t\right)=\res_{\infty}\left(\frac{f_i(t)}{z-t}\dv t\right)=0
\end{align*}
by the residue theorem and the fact that
\begin{align*}
\frac{f_i(t^{-1})}{z-t^{-1}}t^{-2}=\frac{t^{m-1}P_i(t^{-1})}{(tz-1)t^mQ(t^{-1})}
\end{align*}
is bounded in $t$ on the complement of an open ball around $0\in\C$ containing $Z(Q)\cup\{z\}$, as $t^mQ(t^{-1})$ is a polynomial in $t$ with value $1$ at $t=0$ and $\deg(P_i)<m$. Hence
\begin{align*}
g_i(z)=-\res_z\left(\frac{f_i(t)}{z-t}\dv t\right)=f_i(z)\comma
\end{align*}
so that $F=G$ and the above construction gives a right inverse. Likewise one checks that it gives a left inverse.
\end{proof}

\section{Rational maps to complete flag varieties}

In this section we establish a bijective correspondence between $\mathcal{N}_k$ and a moduli space of Nahm complexes as defined by Hurtubise in \cite{hurtubise1989}. This is done by considering $k$-Nahm complexes in a particular normal form, regularizing via a singular gauge transformation, transforming to a form reminiscent of the normal form in \cite{hurtubise1989}, introducing multiple intervals and picking out on each interval a certain block, and finally ``singularizing" via a singular gauge transformation. A schematic visualization of these steps can be found by following the figures in this section.

Let $\varepsilon>0$ be small enough such that $\lambda_N-n\varepsilon>\lambda_1$ and define $\lambda_i(\varepsilon)=\lambda_N-(N-i)\varepsilon$ for $i\in\{2,\dots,n\}$, $\lambda_1(\varepsilon)=\lambda_1$ and $\lambda_N(\varepsilon)=\lambda_N$. Denote by $\mathcal{B}_k$ the set of all  $g_\mathrm{st}.(\alpha_\mathrm{st},\beta_\mathrm{st})$ with $(\cdot-\lambda_N)^{2x_N}\beta(\cdot)(\cdot-\lambda_N)^{-2x_N}$ constant on $(\lambda_N-n\varepsilon,\lambda_N)$, and let $g_\mathrm{reg}\in C^\infty(\lambda,\GL{m}{\C})$ satisfy $g_\mathrm{reg}(\lambda_1+z)=z^{-2x_1}$ for $z>0$ small, and $g_\mathrm{reg}(\lambda_N+z)=z^{2x_N}$ for $z<0$ of small modulus.
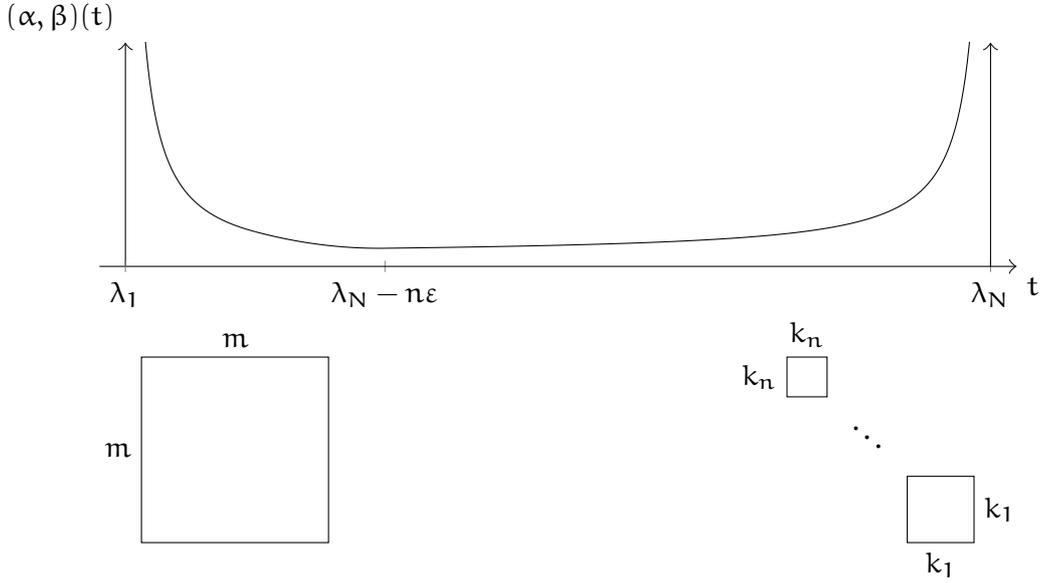
\begin{figure}[H]
\centering
\begin{tikzpicture}
\begin{axis}[
	axis lines=center,
	xmin=-0.3,
	xmax=10.3,
	ymin=0,
	ymax=4.3,
	xlabel={$t$},
	ylabel={$(\alpha,\beta)(t)$},
	xlabel style={below right},
	ylabel style={above left},
	xtick={3,10},
	xticklabels={$\lambda_N-n\varepsilon$,$\lambda_N$},
	extra x ticks={0},
	extra x tick labels={$\lambda_1$},
	ytick={0},
	axis line style={-to new, arrow head=2.8pt},
	width=0.9\textwidth,
	height=0.3\textwidth]
\addplot[mark=none,domain=0.2:1.5,samples=100]{1/x};
\addplot[mark=none,domain=3:9.8,samples=300]{1/(10-x)+0.206};
\addplot[mark=none,domain=1.5:3,samples=30]{1/3*(1/49+4/9)*x^2-(1/49+8/9)*x+2/3-9/4*1/3*(1/49+4/9)+3/2*(1/49+8/9)};
\draw[-to new, arrow head=2.8pt](axis cs:10,0)--(axis cs:10,4.3);
\end{axis}
\draw (0.55,-1.2) +(0pt,-70pt) rectangle +(70pt,0pt) node[anchor=south,xshift=-35pt,yshift=0pt] {$m$} node[anchor=east,xshift=-70pt,yshift=-35pt] {$m$};
%\draw (11.5,-0.8) +(-70pt,-70pt) rectangle +(0pt,0pt);
\draw (9.04,-1.2) +(0pt,-15pt) rectangle +(15pt,0pt) node[anchor=south,xshift=-7.5pt,yshift=0pt] {$k_n$} node[anchor=east,xshift=-15pt,yshift=-7.5pt] {$k_n$} node[xshift=15pt,yshift=-27pt] {$\ddots$};
\draw (11.5,-3.66) +(-25pt,0pt) rectangle +(0pt,25pt) node[anchor=north,xshift=-12.5pt,yshift=-25pt] {$k_1$} node[anchor=west,xshift=0pt,yshift=-12.5pt] {$k_1$};
\end{tikzpicture}
\caption{Schematic visualization of elements of $\mathcal{B}_k$ with the squares corresponding to irreducible blocks of the residue}
\end{figure}
\noindent Then $g_\mathrm{reg}.\mathcal{B}_k$ is the set of solutions $(\alpha,\beta)$ to the complex equation such that $\alpha=0$ near $\lambda_1$ and on $(\lambda_N-n\varepsilon,\lambda_N)$, and with
\begin{align*}
\beta(\lambda_1+z)=\left(\begin{matrix}
0&&&-q_1\\
1&\ddots&&\vdots\\
&\ddots&0&-q_{m-1}\\
&&1&-q_m
\end{matrix}\right)
\end{align*}
for $z>0$ small and $\beta|_{(\lambda_N-n\varepsilon,\lambda_N)}$ constant of the form as $B$ in corollary \ref{bijection of matrices for normal forms}. Note that $g_\mathrm{reg}.:\mathcal{B}_k\to g_\mathrm{reg}.\mathcal{B}_k$ is a bijection regularizing a $k$-Nahm complex via a singular gauge transformation.
\newcommand{\alphabetal}{(0,\left(\tiny\begin{matrix}0&&&{*}\\1&\ddots&&\vdots\\&\ddots&0&{*}\\&&1&{*}\end{matrix}\right))}
\newcommand{\alphabetar}{(0,\left(\begin{matrix}\tiny\begin{matrix}0&&&{*}\\1&\ddots&&\vdots\\&\ddots&0&{*}\\&&1&{*}\end{matrix}&\dots&\tiny\begin{matrix}\phantom{0}&&&&&&{*}\\\phantom{1}&\phantom{\ddots}&&\phantom{0}&&&\vdots\\&\phantom{\ddots}&\phantom{0}&&\phantom{\ddots}&\phantom{0}&{*}\\&&\phantom{1}&&&\phantom{1}&{*}\end{matrix}\\\vdots&\ddots&\vdots\\\tiny\begin{matrix}\phantom{0}&&&{*}\\\phantom{1}&\phantom{\ddots}&&\vdots\\&\phantom{\ddots}&\phantom{0}&{*}\\&&\phantom{1}&{*}\\&&&0\\&&&\vdots\\&&&0\end{matrix}&\dots&\tiny\begin{matrix}0&&&&&&{*}\\1&\ddots&&&&&\vdots\\&\ddots&0&&&&{*}\\&&1&0&&&{*}\\&&&1&\ddots&&{*}\\&&&&\ddots&0&\vdots\\&&&&&1&{*}\end{matrix}\end{matrix}\right))}
\begin{figure}[H]
\centering
\begin{tikzpicture}
\begin{axis}[
	axis lines=center,
	xmin=-0.3,
	xmax=10.3,
	ymin=0,
	ymax=4.3,
	xlabel={$t$},
	ylabel={$(\alpha,\beta)(t)$},
	xlabel style={below right},
	ylabel style={above left},
	xtick={3,10},
	xticklabels={$\lambda_N-n\varepsilon$,$\lambda_N$},
	extra x ticks={0},
	extra x tick labels={$\lambda_1$},
	ytick={0},
	axis line style={-to new, arrow head=2.8pt},
	width=0.9\textwidth,
	height=0.3\textwidth]
\addplot[mark=none,domain=0:1.5,samples=100]{1.5};
\addplot[mark=none,domain=3:10,samples=300]{2};
\addplot[mark=none,domain=1.5:3,samples=30]{-1/4*cos(deg((x-3/2)*2*pi))+7/4};
\draw[-to new, arrow head=2.8pt](axis cs:10,0)--(axis cs:10,4.3);
\path (axis cs:0,0) coordinate (O) (axis cs:1.5,0) coordinate (A) (axis cs:3,0) coordinate (B) (axis cs:10,0) coordinate (R);
\end{axis}
%node at (0,0) {$(0,\left(\right))$};
\draw[decorate,decoration={brace,raise=2em}] (A)--(O) node[midway,below=2.4em]{$\alphabetal$};
\draw[decorate,decoration={brace,raise=2em}] (R)--(B) node[midway,below=2.4em]{$\alphabetar$};
\end{tikzpicture}
\caption{Schematic visualization of elements of $g_\mathrm{reg}.\mathcal{B}_k$}
\end{figure}
\newcommand{\alphabetah}{(0,\left(\begin{matrix}\tiny\begin{matrix}0&&&{*}\\1&\ddots&&\vdots\\&\ddots&0&{*}\\&&1&{*}\end{matrix}&\dots&\tiny\begin{matrix}\phantom{0}&&&&&&{*}\\\phantom{1}&\phantom{\ddots}&&\phantom{0}&&&\vdots\\&\phantom{\ddots}&\phantom{0}&&\phantom{\ddots}&\phantom{0}&{*}\\&&\phantom{1}&&&\phantom{1}&{*}\end{matrix}\\\vdots&\ddots&\vdots\\\tiny\begin{matrix}{*}&\dots&{*}&{*}\\\phantom{1}&\phantom{\ddots}&&0\\&\phantom{\ddots}&\phantom{0}&\vdots\\&&\phantom{1}&0\\&&&0\\&&&\vdots\\&&&0\end{matrix}&\dots&\tiny\begin{matrix}0&&&&&&{*}\\1&\ddots&&&&&\vdots\\&\ddots&0&&&&{*}\\&&1&0&&&{*}\\&&&1&\ddots&&{*}\\&&&&\ddots&0&\vdots\\&&&&&1&{*}\end{matrix}\end{matrix}\right))}
\begin{figure}[H]
\centering
\begin{tikzpicture}
\begin{axis}[
	axis lines=center,
	xmin=-0.3,
	xmax=10.3,
	ymin=0,
	ymax=4.3,
	xlabel={$t$},
	ylabel={$(\alpha,\beta)(t)$},
	xlabel style={below right},
	ylabel style={above left},
	xtick={3,10},
	xticklabels={$\lambda_N-n\varepsilon$,$\lambda_N$},
	extra x ticks={0},
	extra x tick labels={$\lambda_1$},
	ytick={0},
	axis line style={-to new, arrow head=2.8pt},
	width=0.9\textwidth,
	height=0.3\textwidth]
\addplot[mark=none,domain=0:1.5,samples=100]{1.5};
\addplot[mark=none,domain=3:10,samples=300]{1};
\addplot[mark=none,domain=1.5:3,samples=30]{1/4*cos(deg((x-3/2)*2*pi))+5/4};
\draw[-to new, arrow head=2.8pt](axis cs:10,0)--(axis cs:10,4.3);
\path (axis cs:0,0) coordinate (O) (axis cs:1.5,0) coordinate (A) (axis cs:3,0) coordinate (B) (axis cs:10,0) coordinate (R);
\end{axis}
%node at (0,0) {$(0,\left(\right))$};
\draw[decorate,decoration={brace,raise=2em}] (A)--(O) node[midway,below=2.4em]{$\alphabetal$};
\draw[decorate,decoration={brace,raise=2em}] (R)--(B) node[midway,below=2.4em]{$\alphabetah$};
\end{tikzpicture}
\caption{Schematic visualization of elements of $g.g_\mathrm{reg}.\mathcal{B}_k$}
\end{figure}

Further define a bijection $g.:g_\mathrm{reg}.\mathcal{B}_k\to g.g_\mathrm{reg}.\mathcal{B}_k$ by defining $g$ at $(\alpha,\beta)$ to be a gauge transformation that is constant $\idmat$ near $\lambda_1$ and has as a constant value on $(\lambda_N-n\varepsilon,\lambda_N)$ the transformation of corollary \ref{bijection of matrices for normal forms} for $B$ the constant value of $\beta$ near $\lambda_N$. $g.g_\mathrm{reg}.\mathcal{B}_k$ then is the set of all solutions to the complex equation that near $\lambda_1$ are of the same form as the elements of $g_\mathrm{reg}.\mathcal{B}_k$ and on $(\lambda_N-n\varepsilon,\lambda_N)$ have $\alpha=0$ and $\beta$ constant of the form as $B'$ in lemma \ref{matrix to hurtubise normal form surjectivity} for $(i,j)=(m,m-1)$. Thus $g.$ is the action of a gauge transformation that is constant (in $t$) on appropriate neighborhoods of the ends of the interval but that may depend on $\beta$.

\newcommand{\alphabetaa}{(0,\left(\begin{matrix}\tiny\begin{matrix}0&&&{*}\\1&\ddots&&\vdots\\&\ddots&0&{*}\\&&1&{*}\end{matrix}&\dots\quad\\\vdots&\ddots\quad\end{matrix}\right))}

\begin{figure}[H]
\centering
\begin{tikzpicture}
\begin{axis}[
	axis lines=center,
	xmin=-0.3,
	xmax=10.3,
	ymin=0,
	ymax=4.3,
	xlabel={$t$},
	ylabel={$(\alpha,\beta)(t)$},
	xlabel style={below right},
	ylabel style={above left},
	xtick={3,5,7,8,10},
	xticklabels={$\lambda_N-n\varepsilon$,$\lambda_2(\varepsilon)$,$\lambda_3(\varepsilon)$,$\lambda_n(\varepsilon)$,$\lambda_N$},
	extra x ticks={0},
	extra x tick labels={$\lambda_1$},
	ytick={0},
	axis line style={-to new, arrow head=2.8pt},
	width=0.9\textwidth,
	height=0.3\textwidth]
\addplot[mark=none,domain=0:1.5,samples=100]{1.5};
\addplot[mark=none,domain=3:10,samples=300]{1};
\addplot[mark=none,domain=1.5:3,samples=30]{1/4*cos(deg((x-3/2)*2*pi))+5/4};
\draw[-to new, arrow head=2.8pt](axis cs:10,0)--(axis cs:10,4.3);
\draw[-to new, arrow head=2.8pt](axis cs:5,0)--(axis cs:5,4.3);
\draw[-to new, arrow head=2.8pt](axis cs:7,0)--(axis cs:7,4.3);
\draw[-to new, arrow head=2.8pt](axis cs:8,0)--(axis cs:8,4.3);
\path (axis cs:0,0) coordinate (O) (axis cs:1.5,0) coordinate (A) (axis cs:3,0) coordinate (B) (axis cs:5,0) coordinate (R);
\path (axis cs:8,0) coordinate (S) (axis cs:10,0) coordinate (T);
\path (axis cs:5,0) coordinate (U) (axis cs:7,0) coordinate (V);
\draw[fill=white,draw=none](axis cs:7.5,0) coordinate (W) +(-10pt,2pt) rectangle +(10pt,-2pt) node {};
\draw[fill=white,draw=none](axis cs:7.5,1) coordinate (X) +(-10pt,2pt) rectangle +(10pt,-2pt) node {};
\end{axis}
%node at (0,0) {$(0,\left(\right))$};
\draw[decorate,decoration={brace,raise=2em}] (A)--(O) node[midway,below=2.4em]{$\alphabetal$};
\draw[decorate,decoration={brace,raise=2em}] (R)--(B) node[midway,below=8.8em]{$\alphabetah$};
\draw[decorate,decoration={brace,raise=2em}] (T)--(S) node[midway,below=8.8em]{$\alphabetal$};
\draw[decorate,decoration={brace,raise=2em}] (V)--(U) node[midway,below=2.4em]{$\alphabetaa$};
\draw[dashed] (1.85,-3.5) +(53pt,0pt) rectangle +(0pt,-40pt);
\draw[dashed] (1.85,-3.5) +(85pt,0pt) rectangle +(0pt,-62pt);
\draw[dashed] (9.82,-3.5) +(53pt,0pt) rectangle +(0pt,-40pt);
\draw[dashed] (5.85,-1.02) +(85pt,0pt) rectangle +(0pt,-62pt);
\draw[fill=white,draw=none](W) +(-10pt,2pt) rectangle +(10pt,-2pt) node[xshift=-10pt,yshift=2pt] {$\dots$};
\draw[fill=white,draw=none](X) +(-10pt,2pt) rectangle +(10pt,-2pt) node[xshift=-10pt,yshift=2pt] {};
\draw[dotted] (1.85,-3.5)+(53pt,0pt) -- (9.82,-3.5);
\draw[dotted,yshift=-40pt] (1.85,-3.5)+(53pt,0pt) -- (9.82,-3.5);
\draw[dotted] (1.85,-3.5) -- (5.85,-1.02);
\draw[dotted,xshift=85pt,yshift=-62pt] (1.85,-3.5) -- (5.85,-1.02);
\end{tikzpicture}
\caption{Schematic visualization of elements of $(g_\mathrm{reg}^\mathrm{H})^{-1}.\mathcal{B}_k^\mathrm{H}$}
\end{figure}
Define now a bijection $g.g_\mathrm{reg}.\mathcal{B}_k\to (g_\mathrm{reg}^\mathrm{H})^{-1}.\mathcal{B}_k^\mathrm{H}$ by mapping $(\alpha,\beta)$ to
\begin{align*}
\left(\left(\left[\alpha|_{(\lambda_1,\lambda_2(\varepsilon))}\right]_{m_1},\dots,\left[\alpha|_{(\lambda_n(\varepsilon),\lambda_N)}\right]_{m_n}\right),\left(\left[\beta|_{(\lambda_1,\lambda_2(\varepsilon))}\right]_{m_1},\dots,\left[\beta|_{(\lambda_n(\varepsilon,\lambda_N)}\right]_{m_n}\right)\right)\comma
\end{align*}
where $[\cdot]_\ell$ denotes taking the upper left $(\ell\times\ell)$-block, and $\mathcal{B}_k^\mathrm{H}$ is defined to be the image under $g^\mathrm{H}_\mathrm{reg}.$ of the image of the above map with $g^\mathrm{H}_\mathrm{reg}=(g_1,\dots,g_n)$ such that
\begin{align*}
g_1(\lambda_1+z)&=z^{2x(m)}\comma&g_i(\lambda_i(\varepsilon)+z)&=\idmat_{m_i}\comma\\ g_i(\lambda_{i+1}(\varepsilon)+z)&=\diag(\idmat_{m_{i+1}},z^{2x(k_{i})})\comma& g_N(\lambda_N+z)&=z^{2x(k_n)}\fullstop
\end{align*}
With this map, multiple intervals are introduced, and corresponding blocks picked, in order to arrive at a regularized normal form as in \cite{hurtubise1989}. $g_\mathrm{reg}^\mathrm{H}$ then reverses the regularization so that for $((\alpha_1,\dots,\alpha_n),(\beta_1,\dots,\beta_n))\in\mathcal{B}_k^\mathrm{H}$ one obtains a Nahm complex in normal form as defined in \cite{hurtubise1989} by adding the $N$-tuple $(e_1,{v_\mathrm{st}}_2,\dots,{v_\mathrm{st}}_N)$.
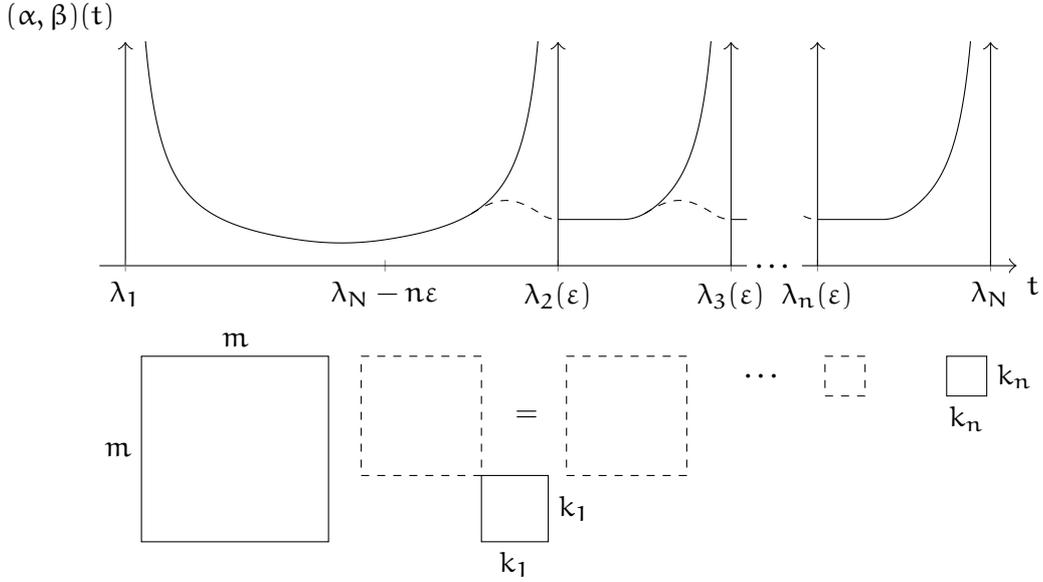
\begin{figure}[H]
\centering
\begin{tikzpicture}
\begin{axis}[
	axis lines=center,
	xmin=-0.3,
	xmax=10.3,
	ymin=0,
	ymax=4.3,
	xlabel={$t$},
	ylabel={$(\alpha,\beta)(t)$},
	xlabel style={below right},
	ylabel style={above left},
	xtick={3,5,7,8,10},
	xticklabels={$\lambda_N-n\varepsilon$,$\lambda_2(\varepsilon)$,$\lambda_3(\varepsilon)$,$\lambda_n(\varepsilon)$,$\lambda_N$},
	extra x ticks={0},
	extra x tick labels={$\lambda_1$},
	ytick={0},
	axis line style={-to new, arrow head=2.8pt},
	width=0.9\textwidth,
	height=0.3\textwidth]
\addplot[mark=none,domain=0.2:2,samples=100]{1/x};
\addplot[mark=none,domain=3:4.8,samples=100]{1/(5-x)};
\addplot[mark=none,domain=2:3,samples=30]{1/4*(x-2.5)^2+7/16};
\addplot[dashed,mark=none,domain=4:4.8,samples=100]{0.8*sin(deg((x-4)*pi/0.8))/pi+1};
\addplot[dashed,mark=none,domain=4.84:5,samples=100]{5*(x-5)^2/2+8/9};
\addplot[mark=none,domain=5:5.75,samples=100]{8/9};
\addplot[mark=none,domain=6.25:6.8,samples=100]{1/(7-x)};
\addplot[mark=none,domain=5.75:6.25,samples=30]{16/9*(x-5.75)^2+8/9};
\addplot[dashed,mark=none,domain=6:6.8,samples=100]{0.8*sin(deg((x-6)*pi/0.8))/pi+1};
\addplot[dashed,mark=none,domain=6.84:7,samples=100]{5*(x-7)^2/2+8/9};
\addplot[mark=none,domain=7:7.185,samples=100]{8/9};
\addplot[mark=none,domain=8:8.75,samples=100]{8/9};
\addplot[mark=none,domain=9.25:9.8,samples=100]{1/(10-x)};
\addplot[mark=none,domain=8.75:9.25,samples=30]{16/9*(x-8.75)^2+8/9};
\addplot[dashed,mark=none,domain=7.84:8,samples=100]{5*(x-8)^2/2+8/9};
\draw[-to new, arrow head=2.8pt](axis cs:10,0)--(axis cs:10,4.3);
\draw[-to new, arrow head=2.8pt](axis cs:5,0)--(axis cs:5,4.3);
\draw[-to new, arrow head=2.8pt](axis cs:7,0)--(axis cs:7,4.3);
\draw[-to new, arrow head=2.8pt](axis cs:8,0)--(axis cs:8,4.3);
\draw[fill=white,draw=none](axis cs:7.5,0) coordinate (W) +(-10pt,2pt) rectangle +(10pt,-2pt) node {$\dots$};
\end{axis}
\draw (0.55,-1.2) +(0pt,-70pt) rectangle +(70pt,0pt) node[anchor=south,xshift=-35pt,yshift=0pt] {$m$} node[anchor=east,xshift=-70pt,yshift=-35pt] {$m$};
%\draw (11.5,-0.8) +(-70pt,-70pt) rectangle +(0pt,0pt);
\draw[dashed] (3.44,-1.2) +(0pt,-45pt) rectangle +(45pt,0pt);
\draw (5.9,-3.66) +(-25pt,0pt) rectangle +(0pt,25pt) node[anchor=north,xshift=-12.5pt,yshift=-25pt] {$k_1$} node[anchor=west,xshift=0pt,yshift=-12.5pt] {$k_1$};
\draw[dashed] (6.14,-1.2) +(0pt,-45pt) rectangle +(45pt,0pt) node[xshift=-60pt,yshift=-22.5pt] {$=$};
\draw[dashed] (9.54,-1.2) +(0pt,-15pt) rectangle +(15pt,0pt);
\draw (11.14,-1.2) +(0pt,-15pt) rectangle +(15pt,0pt) node[anchor=north,xshift=-7.5pt,yshift=-15pt] {$k_n$} node[anchor=west,xshift=0pt,yshift=-7.5pt] {$k_n$};
\draw[fill=white,draw=none](W) +(-10pt,2pt) rectangle +(10pt,-2pt) node[xshift=-10pt,yshift=2pt] {$\dots$};
\draw[fill=white,draw=none](8.73,-1.46) +(-10pt,2pt) rectangle +(10pt,-2pt) node[xshift=-10pt,yshift=2pt] {$\dots$};
\end{tikzpicture}
\caption{Schematic visualization of elements of $\mathcal{B}_k^\mathrm{H}$ with the squares corresponding to irreducible blocks of the residue and dotted squares indicating the value of the corresponding block at the boundary}
\end{figure}

\begin{proposition}\label{correspondence to h nahm complexes}
The above maps induce a bijection between $\mathcal{N}_k$ and the moduli space of Nahm complexes as in \cite{hurtubise1989}, here denoted by $\mathcal{N}_{(m_1,\dots,m_n)}^\mathrm{H}$.\qed
\end{proposition}

\begin{remark}
Permuting the blocks of elements of $g_\mathrm{reg}.\mathcal{B}_k$ by $\sigma\in S(n)$ and then proceeding in analogy to the above construction yields a bijection between $\mathcal{N}_k$ and $\mathcal{N}_{(m_1',\dots,m_n')}^\mathrm{H}$ with
\begin{align*}
m_i'=\sum_{j=1}^nk_{\sigma(j)}\fullstop
\end{align*}
\end{remark}

Using that Hurtubise in \cite{hurtubise1989} established a bijective correspondence between $\mathcal{N}_{(m_1,\dots,m_n)}^\mathrm{H}$ and $\brat{(m_1,\dots,m_n)}{\CP{1}}{F(N)}$, where $F(N)$ denotes the variety of complete flags in $\C^N$, the main result of the present paper may also be obtained as a corollary of the following lemma.

\begin{lemma}
Mapping $F\in R_k$ to the unique flag of subbundles of $E$ in definition \ref{definition based} induces a bijection $\mathrm{Fl}_k:R_k\to\brat{(m_1,\dots,m_n)}{\CP{1}}{F(N)}$, where
\begin{align*}
m_i=\sum_{j=i}^nk_j\comma
\end{align*}
and $\brat{\ell}{\CP{1}}{F(N)}$ denotes the set of based rational maps from $\CP{1}$ to $F(N)$ of degree $\ell=(\ell_1,\dots,\ell_n)$.
\begin{proof}
Let $0\subset E_1\subset\dots\subset E_n\subset E$ be a full flag of subbundles corresponding to an element of $\brat{(m_1,\dots,m_n)}{\CP{1}}{F(N)}$, so that $\deg(E_i)=-m_{N-i}$. There are exact sequences
\[
\begin{tikzcd}
0\arrow{r}{}&E_{i-1}\arrow{r}{}&E_{i}\arrow{r}{}&\faktor{E_{i}}{E_{i-1}}\arrow{r}{}&0
\end{tikzcd}\comma
\]
where $E_i/E_{i-1}\cong\mathcal{O}(-\ell_{N-i})$ for $\ell_i=m_i-m_{i-1}$ with $m_0=m_N=0$. Then,
\begin{align*}
&\mathrm{Ext}^1(\mathcal{O}(-\ell_{N-i}),\bigoplus_{j=1}^{i-1}\mathcal{O}(-\ell_{N-j}))\cong\bigoplus_{j=1}^{i-1}\mathrm{Ext}^1(\mathcal{O}(-\ell_{N-i}),\mathcal{O}(-\ell_{N-j}))\\
&\quad\cong\bigoplus_{j=1}^{i-1}\mathrm{Ext}^1(\mathcal{O},\mathcal{O}(\ell_{N-i}-\ell_{N-j}))\cong\bigoplus_{j=1}^{i-1}H^1(\mathcal{O}(\ell_{N-i}-\ell_{N-j}))=0
\end{align*}
iteratively implies that
\begin{align*}
E_i\cong\bigoplus_{j=1}^i\mathcal{O}(-\ell_{N-j})\fullstop
\end{align*}
In particular, $(E/E_n)^*\cong\mathcal{O}(-m)$ is a line subbundle of $E$ and thus defines a degree $m$ rational map from $\CP{1}$ to $\CP{n}$. This construction induces a map $\brat{(m_1,\dots,m_n)}{\CP{1}}{F(N)}\to R_k$, that since
\begin{align*}
\left(\faktor{E}{\left(\faktor{E}{L(F)}\right)^*}\right)^*=L(F)\subset E\comma&&\left(\faktor{E}{\left(\faktor{E}{E_n}\right)^*}\right)^*=E_n\subset E\comma
\end{align*}
inverts $\mathrm{Fl}_k$.
\end{proof}
\end{lemma}

\section{The symplectic form}

Consider the maps
\[
\begin{tikzcd}
\mathcal{B}_k\arrow{r}{g_\mathrm{reg}.}&g_\mathrm{reg}.\mathcal{B}_k\arrow{r}{g.}&g.g_\mathrm{reg}.\mathcal{B}_k\arrow{r}{}&(g^\mathrm{H}_\mathrm{reg})^{-1}.\mathcal{B}_k^\mathrm{H}\arrow{r}{g_\mathrm{reg}^\mathrm{H}.}&\mathcal{B}_k^\mathrm{H}
\end{tikzcd}
\]
of the previous section. Now identify elements $((\alpha_1,\dots,\alpha_n),(\beta_1,\dots,\beta_n))$ of $\mathcal{B}_k^\mathrm{H}$ with pairs $(\alpha,\beta)$, where $(\alpha,\beta)=(\alpha_i,\beta_i)$ on $(\lambda_i(\varepsilon),\lambda_{i+1}(\varepsilon))$. Then consider on $\mathcal{B}_k^\mathrm{H}$ the symplectic form
\begin{align*}
\omega^{\mathcal{B}_k^\mathrm{H}}=\int_{\lambda_1}^{\lambda_N}\tr(\dv\alpha\wedge\dv\beta)\,\dv t\comma
\end{align*}
i.e. for $(a,b),(a',b')\in T_{(\alpha,\beta)}\mathcal{B}_k^\mathrm{H}$ one has
\begin{align*}
\omega^{\mathcal{B}_k^\mathrm{H}}_{(\alpha,\beta)}((a,b),(a',b'))=\int_{\lambda_1}^{\lambda_n}\tr(ab'-a'b)\,\dv t\fullstop
\end{align*}
Futher denote by $\omega$ with a corresponding superscript the pullback via the above maps to the spaces mentioned above. Note that since the corresponding map $g_\mathrm{reg}^\mathrm{H}:\mathcal{B}_k^\mathrm{H}\to \prod_{i=1}^{n}C^\infty((\lambda_i(\varepsilon),\lambda_{i+1}(\varepsilon)),\GL{m_i}{\C})$ (without the dot) is constant, $\dv g_\mathrm{reg}^\mathrm{H}._{(\alpha,\beta)}$ acts just by conjugation, so that
\begin{align*}
\omega^{(g_\mathrm{reg}^\mathrm{H})^{-1}.\mathcal{B}_k^\mathrm{H}}=\int_{\lambda_1}^{\lambda_N}\tr(\dv\alpha\wedge\dv\beta)\,\dv t\fullstop
\end{align*}
Since $(\alpha,\beta)\in(g_\mathrm{reg}^\mathrm{H})^{-1}.\mathcal{B}_k^\mathrm{H}$ satisfies $\alpha|_{(\lambda_2(\varepsilon),\lambda_N)}=0$, it suffices to integrate over $(\lambda_1,\lambda_2(\varepsilon))$. Hence, as the differential of the map $g.g_\mathrm{reg}.\mathcal{B}_k\to(g_\mathrm{reg}^\mathrm{H})^{-1}.\mathcal{B}_k^\mathrm{H}$ is just given by restriction to $\bigcup_{i=1}^n(\lambda_i(\varepsilon),\lambda_{i+1}(\varepsilon))$, one obtains
\begin{align*}
\omega^{g.g_\mathrm{reg}.\mathcal{B}_k}=\int_{\lambda_1}^{\lambda_N}\tr(\dv\alpha\wedge\dv\beta)\,\dv t\fullstop
\end{align*}

Since the map $g:g_\mathrm{reg}.\mathcal{B}_k\to C^\infty(\lambda,\GL{m}{\C})$ (again without the dot) is not constant, the differential is not just given by conjugation, but one has
\begin{align*}
\dv g._{(\alpha,\beta)}(a,b)=\left(g\left(a-\bar\dv_\alpha(g^{-1}\rho)\right)g^{-1},g\left(b-[\beta,g^{-1}\rho]\right)g^{-1}\right)\comma
\end{align*}
where $\rho=\dv g_{(\alpha,\beta)}(a,b)$,
\begin{align*}
\bar\dv_\alpha=\frac{1}{2}\frac{\dv}{\dv t}+[\alpha,\cdot]
\end{align*}
(as in \cite{donaldson1984}), and for readability $g$ is written instead of $g(\alpha,\beta)$. Now compute in analogy to \cite[p. 78]{bielawski1993} for $(a,b),(a',b')\in T_{(\alpha,\beta)}g_\mathrm{reg}.\mathcal{B}_k$ with $\dv g_{(\alpha,\beta)}(a,b)=\rho$ and $\dv g_{(\alpha,\beta)}(a',b')=\rho'$
\begin{align*}
&\omega^{g_\mathrm{reg}.\mathcal{B}_k}_{(\alpha,\beta)}((a,b),(a',b'))=\int_{\lambda_1}^{\lambda_N}\tr\left(g\left(a-\bar\dv_\alpha(g^{-1}\rho)\right))\left(b'-[\beta,g^{-1}\rho']\right)g^{-1}\right.\\
&\hphantom{\omega^{g_\mathrm{reg}.\mathcal{B}_k}_{(\alpha,\beta)}((a,b),(a',b'))=\int_{\lambda_1}^{\lambda_N}\tr(}\left.{}-{}g\left(a'-\bar\dv_\alpha(g^{-1}\rho')\right))\left(b-[\beta,g^{-1}\rho]\right)g^{-1}\right)\,\dv t\\
&\quad=\int_{\lambda_1}^{\lambda_N}\tr\left(ab'-a'b\right)\,\dv t\\
&\quad\phantom{{}={}}+\int_{\lambda_1}^{\lambda_N}\tr\left(-\bar\dv_\alpha(g^{-1}\rho)b'-a[\beta,g^{-1}\rho']+\bar\dv_\alpha(g^{-1}\rho)[\beta,g^{-1}\rho']\right)\,\dv t\\
&\quad\phantom{{}={}}+\int_{\lambda_1}^{\lambda_N}\tr\left(\bar\dv_\alpha(g^{-1}\rho')b+a'[\beta,g^{-1}\rho]-\bar\dv_\alpha(g^{-1}\rho')[\beta,g^{-1}\rho]\right)\,\dv t\\
&\quad=\int_{\lambda_1}^{\lambda_N}\tr\left(ab'-a'b\right)\,\dv t\\
&\quad\phantom{{}={}}+\int_{\lambda_1}^{\lambda_N}\tr\left(g^{-1}\rho\bar\dv_\alpha b'+g^{-1}\rho'[\beta,a]-g^{-1}\rho\left([\bar\dv_\alpha\beta,\rho']+[\beta,\bar\dv_\alpha\rho']\right)\right)\,\dv t\\
&\quad\phantom{{}={}}+\int_{\lambda_1}^{\lambda_N}\tr\left(-g^{-1}\rho'\bar\dv_\alpha b-g^{-1}\rho[\beta,a']+g^{-1}\rho[\beta,\bar\dv_\alpha(g^{-1}\rho)]\right)\,\dv t\\
&\quad\phantom{{}={}}+\frac{1}{2}\left.\tr\left(-g^{-1}\rho b'+g^{-1}\rho' b+g^{-1}\rho[\beta,g^{-1}\rho']\right)\right|_{\lambda_1}^{\lambda_N}\\
&\quad=\int_{\lambda_1}^{\lambda_N}\tr\left(ab'-a'b\right)\,\dv t+\frac{1}{2}\left.\tr\left(-g^{-1}\rho b'+g^{-1}\rho' b+g^{-1}\rho[\beta,g^{-1}\rho']\right)\right|_{\lambda_1}^{\lambda_N}\comma
\end{align*}
where it was used that the complex equation $\bar\dv_\alpha\beta=0$ yields $\bar\dv_\alpha b=[\beta,a]$ and $\bar\dv_\alpha b'=[\beta,a']$. Since $g$ is constant near $\lambda_1$ independent of $(\alpha,\beta)$ it follows that $\rho(\lambda_1)=0$, so the second term above only needs to be considered at $\lambda_N$. For fixed $(\alpha,\beta)$, $g$ is constant lower triangularly blocked near $\lambda_N$ with diagonal blocks equal to $\idmat$ and all other non-vanishing blocks of the form
\begin{align*}
\left(\begin{matrix}
0&{*}&\dots&{*}\\
&\ddots&\ddots&\vdots\\
&&0&{*}\\
0&\dots&0&0\\
\vdots&\ddots&\vdots&\vdots\\
0&\dots&0&0
\end{matrix}\right)\fullstop
\end{align*}
The same is true for $g^{-1}$. Hence $\rho$ and $\rho'$ are strictly lower triangularly blocked with non-vanishing blocks of that same form. Since
\begin{align*}
\left(\begin{matrix}
0&{*}&\dots&{*}\\
&\ddots&\ddots&\vdots\\
&&0&{*}\\
0&\dots&0&0\\
\vdots&\ddots&\vdots&\vdots\\
0&\dots&0&0
\end{matrix}\right)
\left(\begin{matrix}
0&\dots&0&{*}\\
\vdots&\ddots&\vdots&\vdots\\
0&\dots&0&{*}
\end{matrix}\right)=
\left(\begin{matrix}
0&\dots&0&{*}\\
\vdots&\ddots&\vdots&\vdots\\
0&\dots&0&{*}\\
0&\dots&0&0\\
\vdots&\ddots&\vdots&\vdots\\
0&\dots&0&0
\end{matrix}\right)
\end{align*}
$g^{-1}\rho b'$ and $g^{-1}\rho'b$ at $\lambda_N$ have vanishing diagonal entries, so that their trace is zero. Further, observe that $\beta-y_N$ near $\lambda_N$ is of the same form as $b$ and $b'$, so that $\tr(\rho\rho'(\beta-y_N))=0$, while $(\beta-y_N)b'=0$, as
\begin{align*}
\left(\begin{matrix}
0&\dots&0&{*}\\
\vdots&\ddots&\vdots&\vdots\\
0&\dots&0&{*}
\end{matrix}\right)
\left(\begin{matrix}
0&{*}&\dots&{*}\\
&\ddots&\ddots&\vdots\\
&&0&{*}\\
0&\dots&0&0\\
\vdots&\ddots&\vdots&\vdots\\
0&\dots&0&0
\end{matrix}\right)=0\fullstop
\end{align*}
Finally observe that $[y_N,\rho']$ is again strictly lower triangularly blocked, so that one obtains $\tr(\rho[y_N,\rho'])=0$. Altogether, this shows that at $\lambda_N$
\begin{align*}
\tr(\rho[\beta,\rho'])=\tr(\rho[y_N,\rho'])+\tr(\rho[\beta-y_N,\rho'])=0\comma
\end{align*}
which then yields
\begin{align*}
\frac{1}{2}\left.\tr\left(-g^{-1}\rho b'+g^{-1}\rho' b+g^{-1}\rho[\beta,g^{-1}\rho']\right)\right|_{\lambda_1}^{\lambda_N}=0\comma
\end{align*}
so that
\begin{align*}
\omega^{g_\mathrm{reg}.\mathcal{B}_k}=\int_{\lambda_1}^{\lambda_N}\tr(\dv\alpha\wedge\dv\beta)\,\dv t\fullstop
\end{align*}

Similar to the argument for $g_\mathrm{reg}^\mathrm{H}$ above, note that the differential $\dv g_\mathrm{reg}._{(\alpha,\beta)}$ is just given by conjugation, as $g_\mathrm{reg}:\mathcal{B}_k\to C^\infty(\lambda,\GL{m}{\C})$ is constant, which finally yields
\begin{align*}
\omega^{\mathcal{B}_k}=\int_{\lambda_1}^{\lambda_N}\tr(\dv\alpha\wedge\dv\beta)\,\dv t\fullstop
\end{align*}
This shows

\begin{proposition}
The map of proposition \ref{correspondence to h nahm complexes} from $\mathcal{N}_k$ to $\mathcal{N}_{(m_1,\dots,m_n)}^\mathrm{H}$ preserves the symplectic form given, respectively, by
\[
\pushQED{\qed}
\int_{\lambda_1}^{\lambda_N}\tr(\dv\alpha\wedge\dv\beta)\,\dv t\fullstop\qedhere
\popQED
\]
\end{proposition}

\begin{appendices}
\renewcommand{\thesection}{}

\section[Appendix]{}
\renewcommand{\thesection}{\Alph{section}}

\begin{lemma}\label{parallel transport solutions}
Let $((\alpha,\beta),(e_m^T,{v_{\mathrm{st}}}_2,\dots,{v_\mathrm{st}}_N))$ be a $k$-Nahm complex with $\res_{\lambda_1}(\alpha)=-x_1$ and $\res_{\lambda_1}(\beta)=y_1^T$ and that satisfies (ii) in the definition of a $k$-Nahm complex for $g=\idmat$. Then there exist unique smooth solutions on $\lambda$ to the following problems:
\begin{align}
&\begin{dcases}\dot{u}_1=-2\alpha u_1\\\lim_{z\searrow0}z^{-\frac{m-1}{2}}u_1(\lambda_1+z)=e_1\end{dcases}\label{problem for u2}\\
&\begin{dcases}\dot{u}_{i+1}=-2\alpha u_{i+1}\\\lim_{z\nearrow0}z^{-\frac{k_{N-i}-1}{2}}u_{i+1}(\lambda_N+z)={v_\mathrm{st}}_{i+1}\\
\forall j<N-i:\forall\ell\in\{m_{N-i-j}+1,\dots,m_{n-i-j}\}:\lim_{z\nearrow0}z^{-\frac{k_{N-j}-1}{2}}(u_{i+1})_\ell(\lambda_N+z)=0\end{dcases}\label{problem for u12}
%&\begin{dcases}\dot u_{11}=-2\alpha u_{11}\\\lim_{z\searrow0}z^{-\frac{k_1-1}{2}}u_{11}(\lambda_1+z)={v_\mathrm{st}}_{11}\end{dcases}\comma\\
%&\begin{dcases}\dot u_{12}=-2\alpha u_{12}\\\lim_{z\searrow 0}z^{-\frac{k_2-1}{2}}u_{12}(\lambda_1+z)={v_\mathrm{st}}_{12}\\\forall i\in\{1,\dots,k_1\}:\lim_{z\searrow 0}z^{-\frac{k_1-1}{2}}(u_{12})_i(\lambda_1+z)=0\end{dcases}\comma\\
%&\begin{dcases}\dot u_2=2u_2\alpha\\\lim_{z\nearrow0}z^{-\frac{m-1}{2}}u_2(\lambda_2+z)={v_\mathrm{st}}_2\end{dcases}\fullstop
\end{align}
\begin{proof}
Consider first the case where $m$ is even. Define $A(z)=-4z\sigma\alpha(\lambda_N+z^2)\sigma^{-1}$ with $\sigma$ conjugating $x_N$ to have decreasing diagonal entries and consider the equation
\begin{align*}
z\dot x=A(z)x\comma
\end{align*}
which has a singularity of the first kind at $0$. Using the result \cite[p.32, Theorem 6 and Remark 1]{balser1999} for such ODEs gives the existence of a fundamental matrix $X$ with
\begin{align*}
X(z)=T(z)z^Kz^M\comma
\end{align*}
for $z<0$ of small modulus such that
\begin{enumerate}[label=(\roman*), itemsep=0mm]
\item $T$ is an analytic transformation near $z=0$, i.e. it is analytic near $z=0$ with invertible constant term $T_0$, and $T_0=\idmat$,
\item $M$ is lower triangularly blocked of some type $(\nu_1,\dots,\nu_\mu)$,
\item the $\ell^\text{th}$ diagonal block of $M$ has exactly one eigenvalue $m_\ell$ with real part in the half-open interval $[0,1)$,
\item $K=\diag(\kappa_1\idmat_{\nu_1},\dots,\kappa_\mu\idmat_{\nu_\mu})$, $\kappa_j\in\Z$, weakly increasing and so that $m_\ell+\kappa_\ell$ is an eigenvalue of $-2x_N$ with algebraic multiplicity $\nu_\ell$.
\end{enumerate}
Since $-4\sigma x_N\sigma^{-1}$ has only integer eigenvalues, it follows that all $m_\ell$ vanish. Considering
\begin{align*}
z\dot T(z)T^{-1}(z)=A(z)-T(z)z^KMz^{-K}T^{-1}(z)-T(z)KT^{-1}(z)
\end{align*}
at $z=0$ implies $4\sigma x_N\sigma^{-1}+M+K=0$. Hence the $(i,i)$-block of $\sigma^{-1}K\sigma$ is $-4x(k_i)$ (as defined in section \ref{nahms equation and nahm complexes}), while $M$ vanishes. One obtains $\lim_{z\nearrow 0}X(z)z^{-K}=\idmat$ with $X(z)z^{-K}=\idmat+O(z)$, so that defining $Y(\lambda_N+z)=\sigma^{-1}X(z^{1/2})\sigma$ yields a fundamental matrix of
\begin{align*}
\dot u=-2\alpha u
\end{align*}
satisfying $\lim_{z\nearrow0}Y(\lambda_N+z)z^{2x_N}=\idmat$ with $Y(\lambda_N+z)z^{2x_N}=\idmat+z^{1/2}O(1)$.

Using the fundamental matrix $Y$, set $u_{i+1}=Y_{(\cdot)(m_{N+1-i}+1)}$ to obtain solutions to the problems
\begin{align*}
\begin{dcases}\dot u_{i+1}=-2\alpha u_{i+1}\\\lim_{z\nearrow0}z^{-\frac{k_{N-i}-1}{2}}u_{i+1}(\lambda_N+z)=v_{i+1}\end{dcases}\comma
\end{align*}
where $u_{N}$ is unique already. Since any solution to the equations is a linear combination of columns of $Y$, one can achieve
\begin{align*}
\forall j<N-i:\forall\ell\in\{m_{N-i-j}+1,\dots,m_{n-i-j}\}:\lim_{z\nearrow0}z^{-\frac{k_{N-j}-1}{2}}(u_{i+1})_\ell(\lambda_N+z)=0
\end{align*}
by adding to $u_{i+1}$ a linear combination of columns of $Y$. This given existence and uniqueness of a solution to the problem (\ref{problem for u12}). A similar argument yields the result for the case where $m$ is odd.

Existence and uniqueness of a solution to the problem (\ref{problem for u2}) is shown analogously.
\end{proof}
\end{lemma}

\begin{lemma}\label{conjugation to normal form}
\begin{align*}
&\left(\small\begin{matrix}
z^{m-1}q_{2}&\dots&zq_m&1\\
\vdots&\iddots&1\\
zq_m&\iddots\\
1
\end{matrix}\right)
\left(\small\begin{matrix}0&z^{-1}\\
&\ddots&\ddots\\
&&0&z^{-1}\\
-z^{m-1}q_1&\dots&-zq_{m-1}&-q_m
\end{matrix}\right)
\left(\small\begin{matrix}
z^{m-1}q_{2}&\dots&zq_m&1\\
\vdots&\iddots&1\\
zq_m&\iddots\\
1
\end{matrix}\right)^{-1}\\
&\quad=\left(\small\begin{matrix}
0&&&-z^{m-1}q_1\\
z^{-1}&\ddots&&\vdots\\
&\ddots&0&-zq_{m-1}\\
&&z^{-1}&-q_m
\end{matrix}\right)\fullstop
\end{align*}
\begin{proof}
Compute
\begin{align*}
&\left(\small\begin{matrix}
z^{m-1}q_{2}&\dots&zq_m&1\\
\vdots&\iddots&1\\
zq_m&\iddots\\
1
\end{matrix}\right)
\left(\small\begin{matrix}0&z^{-1}\\
&\ddots&\ddots\\
&&0&z^{-1}\\
-z^{m-1}q_1&\dots&-zq_{m-1}&-q_m
\end{matrix}\right)\\
&\quad=\left(\small\begin{matrix}
-z^{m-1}q_1&0&\dots&0&0\\
0&z^{m-3}q_3&\dots&q_m&z^{-1}\\
\vdots&\vdots&\iddots&z^{-1}\\
0&q_m&\iddots\\
0&z^{-1}
\end{matrix}\right)
\end{align*}
and
\begin{align*}
&\left(\small\begin{matrix}
0&&&-z^{m-1}q_1\\
z^{-1}&\ddots&&\vdots\\
&\ddots&0&-zq_{m-1}\\
&&z^{-1}&-q_m
\end{matrix}\right)
\left(\small\begin{matrix}
z^{m-1}q_{2}&\dots&zq_m&1\\
\vdots&\iddots&1\\
zq_m&\iddots\\
1
\end{matrix}\right)\\
&\quad=\left(\small\begin{matrix}
-z^{m-1}q_1&0&\dots&0&0\\
0&z^{m-3}q_3&\dots&q_m&z^{-1}\\
\vdots&\vdots&\iddots&z^{-1}\\
0&q_m&\iddots\\
0&z^{-1}
\end{matrix}\right)\comma
\end{align*}
so that
\begin{align*}
\left|\small\begin{matrix}
z^{m-1}q_{2}&\dots&zq_m&1\\
\vdots&\iddots&1\\
zq_m&\iddots\\
1
\end{matrix}\right|=\pm 1
\end{align*}
implies the claim.
\end{proof}
\end{lemma}

\begin{lemma}\label{power -1/2}
If $((\alpha,\beta),v)$ is a $k$-Nahm complex, then there exists an equivalent $k$-Nahm complex $((\alpha',\beta'),v')$ such that the power $-\frac{1}{2}$-term in the expansion of $F(\alpha',\beta')$ near $\lambda_N$ vanishes.
\begin{proof}
The only case where power $-\frac{1}{2}$-terms in $F(\alpha,\beta)$ may appear in the $(i,j)$-block and the $(j,i)$-block for $k_j-k_i=1$. Consider first only the case $n=2$ and $k_1-k_2=1$. By proposition \ref{normal form}, assume that $(\alpha,\beta)$ is in normal form with $(\cdot-\lambda_N)^{-2x_N}.(0,\beta_0)=(\alpha,\beta)$ near $\lambda_N$ and
\begin{align*}
\beta_0=\left(\small\begin{matrix}0&&&&C_{11}&&&&&C_{12}\\
1&0\\
&\ddots&\ddots&&\vdots&&&&&\vdots\\
&&1&0\\
&&&1&C_{k_21}&&&&&C_{k_22}\\
&&&&&0\\
&&&&\vdots&1&0&&&\vdots\\
&&&&&&\ddots&\ddots\\
&&&&C_{(m-1)1}&&&1&0&C_{(m-1)2}\\
&&&&0&&&&1&C_{m2}\end{matrix}\right)\fullstop
\end{align*}
Denoting by $E_{ij}$ the $m\times m$ matrix with $(E_{ij})_{i'j'}=\delta_{ii'}\delta_{jj'}$, define smooth $\GL{m}{\C}$-valued functions $g$ and $g'$ on $\lambda$ with $g\equiv\idmat$ and $g'\equiv\idmat$ near $\lambda_1$ and
\begin{align*}
&g(\lambda_N+z)=\idmat+\sum_{i=1}^{k_2-1}\mu_iE_{(m+1-i)(k_2-i)}z^\frac{3}{2}+\sum_{i=1}^{k_1-2}\nu_iE_{(k_2+1-i)(m-1-i)}z^\frac{3}{2}\comma\\
&g'(\lambda_1+z)\\
&\quad=\idmat+\sum_{j=1}^{\floor*{\frac{k_1}{3}}}\sum_{i=1}^{k_1-3j}\left(\prod_{l=0}^{j-1}(-\mu_{i+3l})\right)\left(\prod_{o=0}^{j-1}(-\nu_{i+1+3o})\right)E_{(m+1-i)(m+1-i-3j)}z^{3j}\\
&\phantom{\quad{}=\idmat}+\sum_{j=0}^{\floor*{\frac{k_1}{3}}-1}\sum_{i=1}^{k_1-2-3j}\left(\prod_{l=0}^{j}(-\mu_{i+3l})\right)\left(\prod_{o=0}^{j-1}(-\nu_{i+1+3o})\right)E_{(m+1-i)(k_2-i-3j)}z^{3j+\frac{3}{2}}\\
&\phantom{\quad{}=\idmat}+\sum_{j=0}^{\floor*{\frac{k_1}{3}}-1}\sum_{i=1}^{k_1-2-3j}\left(\prod_{l=0}^{j-1}(-\mu_{i+2+3l})\right)\left(\prod_{o=0}^j(-\nu_{i+3o})\right)E_{(k_2+1-i)(m-1-i-3j)}z^{3j+\frac{3}{2}}\\
&\phantom{\quad{}=\idmat}+\sum_{j=1}^{\floor*{\frac{k_1}{3}}}\sum_{i=1}^{k_2-3j}\left(\prod_{l=0}^{j-1}(-\mu_{i+2+3l})\right)\left(\prod_{o=0}^{j-1}(-\nu_{i+3o})\right)E_{(k_2+1-i)(k_2+1-i-3j)}z^{3j}
\end{align*}
for $z<0$ of small modulus and for to be determined constants $\mu_i\in\C$ and $\nu_i\in\C$. Compute for $z<0$ of small modulus
\begin{align*}
&g(\lambda_N+z)g'(\lambda_N+z)\\
&=\idmat+\sum_{j=1}^{\floor*{\frac{k_1}{3}}}\sum_{i=1}^{k_1-3j}\left(\prod_{l=0}^{j-1}(-\mu_{i+3j})\right)\left(\prod_{o=0}^{j-1}(-\nu_{i+1+3o})\right)E_{(m+1-i)(m+1-i-3j)}z^{3j}\\
&\phantom{{}={}}+\sum_{j=1}^{\floor*{\frac{k_1}{3}}-1}\sum_{i=1}^{k_1-2-3j}\left(\prod_{l=0}^j(-\mu_{i+3l})\right)\left(\prod_{o=0}^{j-1}(-\nu_{i+1+3o})\right)E_{(m+1-i)(k_2-i-3j)}z^{3j+\frac{3}{2}}\\
&\phantom{{}={}}+\sum_{j=1}^{\floor*{\frac{k_1}{3}}-1}\sum_{i=1}^{k_1-2-3j}\left(\prod_{l=0}^{j-1}(-\mu_{i+2+3l})\right)\left(\prod_{o=0}^j(-\nu_{i+3o})\right)E_{(k_2+1-i)(m-1-i-3j)}z^{3j+\frac{3}{2}}\\
&\phantom{{}={}}+\sum_{j=1}^{\floor*{\frac{k_1}{3}}}\sum_{i=1}^{k_2-3j}\left(\prod_{l=0}^{j-1}(-\mu_{i+2+3l})\right)\left(\prod_{o=0}^{j-1}(-\nu_{i+3o})\right)E_{(k_2+1-i)(k_2+1-i-3j)}z^{3j}\\
&\phantom{{}={}}+\sum_{j=1}^{\floor*{\frac{k_1}{3}}-1}\sum_{i=3}^{k_1-3j}\nu_{i-2}\left(\prod_{l=0}^{j-1}(-\mu_{i+3j})\right)\left(\prod_{o=0}^{j-1}(-\nu_{i+1+3o})\right)E_{(k_2+1-i+2)(m+1-i-3j)}z^{3j+\frac{3}{2}}\\
&\phantom{{}={}}+\sum_{j=0}^{\floor*{\frac{k_1}{3}}-1}\sum_{i=3}^{k_1-2-3j}\nu_{i-2}\left(\prod_{l=0}^j(-\mu_{i+3l})\right)\left(\prod_{o=0}^{j-1}(-\nu_{i+1+3o})\right)E_{(k_2+1-i+2)(k_2-i-3j)}z^{3j+3}\\
&\phantom{{}={}}+\sum_{j=0}^{\floor*{\frac{k_1}{3}}-1}\sum_{i=2}^{k_1-2-3j}\mu_{i-1}\left(\prod_{l=0}^{j-1}(-\mu_{i+2+3l})\right)\left(\prod_{o=0}^j(-\nu_{i+3o})\right)E_{(m+2-i)(m-1-i-3j)}z^{3j+3}\\
&\phantom{{}={}}+\sum_{j=1}^{\floor*{\frac{k_1}{3}}-1}\sum_{i=2}^{k_2-3j}\mu_{i-1}\left(\prod_{l=0}^{j-1}(-\mu_{i+2+3l})\right)\left(\prod_{o=0}^{j-1}(-\nu_{i+3o})\right)E_{(m+2-i)(k_2+1-i-3j)}z^{3j+\frac{3}{2}}\\
&=\idmat+\sum_{j=1}^{\floor*{\frac{k_1}{3}}}\sum_{i=1}^{k_1-3j}\left(\prod_{l=0}^{j-1}(-\mu_{i+3j})\right)\left(\prod_{o=0}^{j-1}(-\nu_{i+1+3o})\right)E_{(m+1-i)(m+1-i-3j)}z^{3j}\\
&\phantom{{}={}}+\sum_{j=1}^{\floor*{\frac{k_1}{3}}-1}\sum_{i=1}^{k_1-2-3j}\left(\prod_{l=0}^j(-\mu_{i+3l})\right)\left(\prod_{o=0}^{j-1}(-\nu_{i+1+3o})\right)E_{(m+1-i)(k_2-i-3j)}z^{3j+\frac{3}{2}}\\
&\phantom{{}={}}+\sum_{j=1}^{\floor*{\frac{k_1}{3}}-1}\sum_{i=1}^{k_1-2-3j}\left(\prod_{l=0}^{j-1}(-\mu_{i+2+3l})\right)\left(\prod_{o=0}^j(-\nu_{i+3o})\right)E_{(k_2+1-i)(m-1-i-3j)}z^{3j+\frac{3}{2}}\\
&\phantom{{}={}}+\sum_{j=1}^{\floor*{\frac{k_1}{3}}}\sum_{i=1}^{k_2-3j}\left(\prod_{l=0}^{j-1}(-\mu_{i+2+3l})\right)\left(\prod_{o=0}^{j-1}(-\nu_{i+3o})\right)E_{(k_2+1-i)(k_2+1-i-3j)}z^{3j}\\
&\phantom{{}={}}+\sum_{j=1}^{\floor*{\frac{k_1}{3}}-1}\sum_{i=1}^{k_1-2-3j}\nu_{i}\left(\prod_{l=0}^{j-1}(-\mu_{i+2+3j})\right)\left(\prod_{o=1}^{j}(-\nu_{i+3o})\right)E_{(k_2+1-i)(m-1-i-3j)}z^{3j+\frac{3}{2}}\\
&\phantom{{}={}}+\sum_{j=1}^{\floor*{\frac{k_1}{3}}}\sum_{i=1}^{k_2-3j}\nu_{i}\left(\prod_{l=0}^{j-1}(-\mu_{i+2+3l})\right)\left(\prod_{o=1}^{j-1}(-\nu_{i+3o})\right)E_{(k_2+1-i)(k_2+1-i-3j)}z^{3j}\\
&\phantom{{}={}}+\sum_{j=1}^{\floor*{\frac{k_1}{3}}}\sum_{i=1}^{k_1-3j}\mu_{i}\left(\prod_{l=1}^{j-1}(-\mu_{i+3l})\right)\left(\prod_{o=0}^{j-1}(-\nu_{i+1+3o})\right)E_{(m+1-i)(m+1-i-3j)}z^{3j}\\
&\phantom{{}={}}+\sum_{j=1}^{\floor*{\frac{k_1}{3}}-1}\sum_{i=1}^{k_1-2-3j}\mu_{i}\left(\prod_{l=1}^{j}(-\mu_{i+3l})\right)\left(\prod_{o=0}^{j-1}(-\nu_{i+3o})\right)E_{(m+1-i)(k_2-i-3j)}z^{3j+\frac{3}{2}}\\
&=\idmat\comma
\end{align*}
so that $g'=g^{-1}$ near $\lambda_N$. Set now $(\alpha',\beta')=g.(\alpha,\beta)$, so that
\begin{align*}
&\alpha'(\lambda_N+z)=\frac{x_N}{z}+\sum_{i=1}^{k_2-1}\mu_iE_{(m+1-i)(k_2-i)}x_Nz^{\frac{1}{2}}+\sum_{i=1}^{k_1-2}\nu_iE_{(k_2+1-i)(m-1-i)}x_Nz^{\frac{1}{2}}\\
&\phantom{\alpha'(\lambda_N+z)={}}+\sum_{i=1}^{k_1-2}(-\mu_i)x_NE_{(m+1-i)(k_2-i)}z^{\frac{1}{2}}+\sum_{i=1}^{k_1-2}(-\nu_i)x_NE_{(k_2+1-i)(m-1-i)}z^{\frac{1}{2}}\\
&\phantom{\alpha'(\lambda_N+z)={}}-\frac{1}{2}\sum_{i=1}^{k_2-1}\frac{3\mu_i}{2}E_{(m+1-i)(k_2-i)}z^{\frac{1}{2}}-\frac{1}{2}\sum_{i=1}^{k_1-2}\frac{3\nu_i}{2}E_{(k_2+1-i)(m-1-i)}z^{\frac{1}{2}}+O(z^\frac{3}{2})\\
&\quad=\frac{a_1}{z}+\sum_{i=1}^{k_1-2}\left(\frac{k_2-1}{4}-\frac{i}{2}-\frac{k_1-1}{4}+\frac{i-1}{2}-\frac{3}{4}\right)\mu_iE_{(m+1-i)(k_2-i)}z^{\frac{1}{2}}\\
&\quad\phantom{{}={}}+\sum_{i=1}^{k_1-2}\left(\frac{k_1-1}{4}-\frac{i+1}{2}-\frac{k_2-1}{4}+\frac{i-1}{2}-\frac{3}{4}\right)\nu_iE_{(k_2+1-i)(m-1-i)}z^{\frac{1}{2}}\\
&\quad=\frac{a_1}{z}+\sum_{i=1}^{k_1-2}\left(-\frac{3\mu_i}{2}\right)E_{(m+1-i)(k_2-i)}z^{\frac{1}{2}}+\sum_{i=1}^{k_1-2}\left(-\frac{3\nu_i}{2}\right)E_{(k_2+1-i)(m-1-i)}z^{\frac{1}{2}}+O(z^{\frac{3}{2}})\comma\\
&\beta'(\lambda_N+z)=\frac{y_N}{z}+\sum_{i=1}^{k_2-1}\mu_iE_{(m+1-i)(k_2-i)}y_Nz^{\frac{1}{2}}+\sum_{i=1}^{k_1-2}\nu_iE_{(k_2+1-i)(m-1-i)}y_Nz^{\frac{1}{2}}\\
&\phantom{\beta'(\lambda_N+z)={}}+\sum_{i=1}^{k_1-2}(-\mu_i)y_NE_{(m+1-i)(k_2-i)}z^{\frac{1}{2}}+\sum_{i=1}^{k_1-2}(-\nu_i)y_NE_{(k_2+1-i)(m-1-i)}z^{\frac{1}{2}}\\
&\phantom{\beta'(\lambda_N+z)={}}+C_{(m-1)1}E_{(m-1)k_2}z^{\frac{1}{2}}+C_{k_22}E_{k_2m}z^{\frac{1}{2}}+O(z^{\frac{3}{2}})\\
&\quad=\frac{y_N}{z}+\sum_{i=1}^{k_1-3}(\mu_i-\mu_{i+1})E_{(m+1-i)(k_2-i-1)}z^{\frac{1}{2}}+\sum_{i=1}^{k_1-3}(\nu_i-\nu_{i+1})E_{(k_2+1-i)(m-1-i-1)}z^{\frac{1}{2}}\\
&\quad\phantom{{}={}}+C_{(m-1)1}E_{(m-1)k_2}z^{\frac{1}{2}}+C_{k_22}E_{k_2m}z^{\frac{1}{2}}+O(z^{\frac{3}{2}})\comma
\end{align*}
for $z<0$ of small modulus. Denoting the corresponding power $\frac{1}{2}$-coefficients by $\alpha'_{1/2}$ and $\beta'_{1/2}$, the power $-\frac{1}{2}$-coefficient $F(\alpha',\beta')_{-1/2}$ in $F(\alpha',\beta')(\lambda_N+z)$ is given by
\begin{align*}
&F(\alpha',\beta')_{-1/2}=\frac{\alpha'_{1/2}+\alpha'^*_{1/2}}{2}+2\left(\big[x_N,\alpha'^*_{1/2}\big]+\big[\alpha'_{1/2},x_N^*\big]+\big[y_N,\beta'^*_{1/2}\big]+\big[\beta'_{1/2},y_N^*\big]\right)\\
&\quad=\left(\frac{\alpha'_{1/2}}{2}+2\left(\big[\alpha'_{1/2},x_N^*\big]+\big[\beta'_{1/2},y_N^*\big]\right)\right)+\left(\frac{\alpha'_{1/2}}{2}+2\left(\big[\alpha'_{1/2},x_N^*\big]+\big[\beta'_{1/2},y_N^*\big]\right)\right)^*\comma
\end{align*}
where
\begin{align*}
&\frac{\alpha'_{1/2}}{2}+2\left(\big[\alpha'_{1/2},x_N^*\big]+\big[\beta'_{1/2},y_N^*\big]\right)\\
&\quad=\sum_{i=1}^{k_1-2}\left(-\frac{3\mu_i}{4}\right)E_{(m+1-i)(k_2-i)}+\sum_{i=1}^{k_1-2}\left(-\frac{3\nu_i}{4}\right)E_{(k_2+1-i)(m-1-i)}\\
&\quad\phantom{{}={}}+\sum_{i=1}^{k_1-2}\left(-3\mu_i\right)\big[E_{(m+1-i)(k_2-i)},x_N^*\big]+\sum_{i=1}^{k_1-2}\left(-3\nu_i\right)\big[E_{(k_2+1-i)(m-1-i)},x_N^*\big]\\
&\quad\phantom{{}={}}+\sum_{i=1}^{k_1-3}2(\mu_i-\mu_{i+1})\big[E_{(m+1-i)(k_2-i-1)},y_N^*\big]\\
&\quad\phantom{{}={}}+\sum_{i=1}^{k_1-3}2(\nu_i-\nu_{i+1})\big[E_{(k_2+1-i)(m-1-i-1)},y_N^*\big]\\
&\quad\phantom{{}={}}+2C_{(m-1)1}\big[E_{(m-1)k_2},y_N^*\big]+2C_{k_22}\big[E_{k_2m},y_N^*\big]\\
&\quad=\sum_{i=1}^{k_1-2}\left(-\frac{3}{4}+\frac{9}{4}\right)\mu_iE_{(m+1-i)(k_2-i)}+\sum_{i=1}^{k_1-2}\left(-\frac{3}{4}+\frac{9}{4}\right)\nu_iE_{(k_2+1-i)(m-1-i)}\\
&\quad\phantom{{}={}}+\sum_{i=2}^{k_1-2}2(-\mu_{i+1}+2\mu_i-\mu_{i-1})E_{(m+1-i)(k_2-i)}\\
&\quad\phantom{{}={}}+2(\mu_{k_1-2}-\mu_{k_1-3})E_{(k_2+3)1}+2(\mu_1-\mu_2)E_{m(k_2-1)}\\
&\quad\phantom{{}={}}+\sum_{i=2}^{k_1-2}2(-\nu_{i+1}+2\nu_i-2\nu_{i-1})E_{(k_2+1-i)(m-1-i)}\\
&\quad\phantom{{}={}}+2(\nu_{k_1-2}-\nu_{k_1-3})E_{2(k_2+1)}+2(\nu_1-\nu_2)E_{k_2(m-2)}\\
&\quad\phantom{{}={}}-2C_{(m-1)1}E_{(m-2)k_2}-2C_{k_22}E_{(k_2-1)m}\\
&\quad=\sum_{i=2}^{k_1-3}\left(-2\mu_{i-1}+\frac{11}{2}\mu_i-2\mu_{i+1}\right)E_{(m+1-i)(k_2-i)}\\
&\quad\phantom{{}={}}+\left(\frac{7}{2}\mu_1-2\mu_2\right)E_{m(k_2-1)}+\left(\frac{7}{2}\mu_{k_1-2}-\mu_{k_1-3}\right)E_{(k_2+3)1}\\
&\quad\phantom{{}={}}+\sum_{i=2}^{k_1-3}\left(-2\nu_{i-1}+\frac{11}{2}\nu_i-2\nu_{i+1}\right)E_{(k_2+1-i)(m-i)}\\
&\quad\phantom{{}={}}+\left(\frac{7}{2}\nu_1-2\nu_2\right)E_{k_2(m-1)}+\left(\frac{7}{2}\nu_{k_1-2}-2\nu_{k_1-3}\right)E_{2(k_2+1)}\\
&\quad\phantom{{}={}}-2C_{(m-1)1}E_{(m-2)k_2}-2C_{k_22}E_{(k_2-1)m}\fullstop
\end{align*}
Hence,
\begin{align*}
&F(\alpha',\beta')_{-1/2}\\
&\quad=\sum_{i=2}^{k_1-3}\left(-2\mu_{i-1}+\frac{11}{2}\mu_i-2\mu_{i+1}\right)E_{(m+1-i)(k_2-i)}\\
&\quad\phantom{{}={}}+\left(\frac{7}{2}\mu_1-2\mu_2\right)E_{m(k_2-1)}+\left(\frac{7}{2}\mu_{k_1-2}-\mu_{k_1-3}\right)E_{(k_2+3)1}\\
&\quad\phantom{{}={}}+\sum_{i=2}^{k_1-3}\left(-2\nu_{i-1}+\frac{11}{2}\nu_i-2\nu_{i+1}\right)E_{(k_2+1-i)(m-i)}\\
&\quad\phantom{{}={}}+\left(\frac{7}{2}\nu_1-2\nu_2\right)E_{k_2(m-1)}+\left(\frac{7}{2}\nu_{k_1-2}-2\nu_{k_1-3}\right)E_{2(k_2+1)}\\
&\quad\phantom{{}={}}-2C_{(m-1)1}E_{(m-2)k_2}-2C_{k_22}E_{(k_2-1)m}\\
&\quad\phantom{{}={}}+\sum_{i=2}^{k_1-3}\left(-2\bar\mu_{i-1}+\frac{11}{2}\bar\mu_i-2\bar\mu_{i+1}\right)E_{(k_2-i)(m+1-i)}\\
&\quad\phantom{{}={}}+\left(\frac{7}{2}\bar\mu_1-2\bar\mu_2\right)E_{(k_2-1)m}+\left(\frac{7}{2}\bar\mu_{k_1-2}-\bar\mu_{k_1-3}\right)E_{1(k_2+3)}\\
&\quad\phantom{{}={}}+\sum_{i=2}^{k_1-3}\left(-2\bar\nu_{i-1}+\frac{11}{2}\bar\nu_i-2\bar\nu_{i+1}\right)E_{(m-1-i)(k_2+1-i)}\\
&\quad\phantom{{}={}}+\left(\frac{7}{2}\bar\nu_1-2\bar\nu_2\right)E_{(m-1-1)k_2}+\left(\frac{7}{2}\bar\nu_{k_1-2}-2\bar\nu_{k_1-3}\right)E_{(k_2+1)2}\\
&\quad\phantom{{}={}}-2\bar C_{(m-1)1}E_{k_2(m-2)}-2\bar C_{k_22}E_{m(k_2-1)}\\
&\quad=\sum_{i=2}^{k_1-3}\left(-2\mu_{i-1}+\frac{11}{2}\mu_i-2\mu_{i+1}\right)E_{(m+1-i)(k_2-i)}\\
&\quad\phantom{{}={}}+\sum_{i=2}^{k_1-3}\left(-2\bar\mu_{i-1}+\frac{11}{2}\bar\mu_i-2\bar\mu_{i+1}\right)E_{(k_2-i)(m+1-i)}\\
&\quad\phantom{{}={}}+\left(\frac{7}{2}\mu_1-2\mu_2-2\bar C_{k_22}\right)E_{m(k_2-1)}+\left(\frac{7}{2}\mu_{k_1-2}-\mu_{k_1-3}\right)E_{(k_2+3)1}\\
&\quad\phantom{{}={}}+\left(\frac{7}{2}\bar\mu_1-2\bar\mu_2-2C_{k_22}\right)E_{(k_2-1)m}+\left(\frac{7}{2}\bar\mu_{k_1-2}-\bar\mu_{k_1-3}\right)E_{1(k_2+3)}\\
&\quad\phantom{{}={}}+\sum_{i=2}^{k_1-3}\left(-2\nu_{i-1}+\frac{11}{2}\nu_i-2\nu_{i+1}\right)E_{(k_2+1-i)(m-1-i)}\\
&\quad\phantom{{}={}}+\sum_{i=2}^{k_1-3}\left(-2\bar\nu_{i-1}+\frac{11}{2}\bar\nu_i-2\bar\nu_{i+1}\right)E_{(m-1-i)(k_2+1-i)}\\
&\quad\phantom{{}={}}+\left(\frac{7}{2}\nu_1-2\nu_2-2\bar C_{(m-1)1}\right)E_{k_2(m-2)}+\left(\frac{7}{2}\nu_{k_1-2}-2\nu_{k_1-3}\right)E_{2(k_2+1)}\\
&\quad\phantom{{}={}}+\left(\frac{7}{2}\bar\nu_1-2\bar\nu_2-2C_{(m-1)1}\right)E_{(m-2)k_2}+\left(\frac{7}{2}\bar\nu_{k_1-2}-2\bar\nu_{k_1-3}\right)E_{(k_2+1)2}\comma
\end{align*}
so that the vanishing of this expression is equivalent to $\mu_i$ and $\nu_i$ solving
\[
\left\lbrace\begin{aligned}
\frac{7\mu_1}{2}-2\mu_2&=-2\bar{C}_{k_22}\\
-2\mu_1+\frac{11\mu_2}{2}-\mu_3&=0\\
&\:\vdots\\
-2\mu_{k_1-4}+\frac{11\mu_{k_1-3}}{2}-2\mu_{k_1-2}&=0\\
-2\mu_{k_1-3}+\frac{7\mu_{k_1-2}}{2}&=0
\end{aligned}\right.
\]
and
\[
\left\lbrace\begin{aligned}
\frac{7\nu_1}{2}-2\nu_2&=-2\bar{C}_{(m-1)1}\\
-\nu_1+\frac{11\nu_2}{2}-2\nu_3&=0\\
&\:\vdots\\
-\nu_{k_1-4}+\frac{11\nu_{k_1-3}}{2}-2\nu_{k_1-2}&=0\\
-2\nu_{k_1-3}+\frac{\nu_{k_1-2}}{2}&=0
\end{aligned}\right.\fullstop
\]
Denote now for $c\in\R$ and $\ell\in\N$ by $M_\ell(c)$ the $\ell\times\ell$ matrix given by
\begin{align*}
M_\ell(c)=\left(\begin{matrix}
c&-2\\
-2&\frac{11}{2}&-2\\
&-2&\frac{11}{2}&\ddots\\
&&\ddots&\ddots&-2\\
&&&-2&\frac{11}{2}&-2\\
&&&&-2&c
\end{matrix}\right)
\end{align*}
and consider for $\ell>3$
\begin{align*}
&\det(M_\ell(c))=\abs{\begin{matrix}
c&0&-2&0&\dots&0\\
0&c&0&\dots&0&-2\\
-2&0&\frac{11}{2}&-2\\
0&\vdots&-2&\frac{11}{2}&\ddots\\
\vdots&0&&\ddots&\ddots&-2\\
0&-2&&&-2&\frac{11}{2}\end{matrix}}\\
&\quad=\abs{\begin{matrix}c&0\\0&c\end{matrix}}\cdot\abs{\left(\begin{matrix}\frac{11}{2}&-2\\-2&\frac{11}{2}&\ddots\\&\ddots&\ddots&-2\\&&-2&\frac{11}{2}\end{matrix}\right)-\left(\begin{matrix}-2&0\\0&\vdots\\&0\\0&-2\end{matrix}\right)\left(\begin{matrix}\frac{1}{c}&0\\0&\frac{1}{c}\end{matrix}\right)\left(\begin{matrix}-2&0&\dots&0\\0&\dots&0&-2\end{matrix}\right)}\\
&=c^2\det(M_{\ell-2}(\tfrac{11}{2}-\tfrac{4}{c}))\fullstop
\end{align*}
Set $c_1=\frac{7}{2}$ and $c_i=\frac{11}{2}-\frac{4}{c_{i-1}}$ for $i\in\{2,\dots,\floor*{\frac{\ell}{2}}\}$, so that iteratively one obtains
\begin{align*}
\det(M_\ell(c_1))&=\begin{dcases}\left|\begin{matrix}c_{\floor*{\frac{\ell}{2}}}&-2\\-2&c_{\floor*{\frac{\ell}{2}}}\end{matrix}\right|\cdot\prod_{i=1}^{\floor*{\frac{\ell}{2}}-1}c_i^2&\text{, if }\ell\in 2\N\\\left|\begin{matrix}c_{\floor*{\frac{\ell}{2}}}&-2\\-2&\frac{11}{2}&-2\\&-2&c_{\floor*{\frac{\ell}{2}}}\end{matrix}\right|\cdot\prod_{i=1}^{\floor*{\frac{\ell}{2}}-1}c_i^2&\text{, if }\ell\in 2\N+1\end{dcases}\\
&=\begin{dcases}\left(c_{\floor*{\frac{\ell}{2}}}^2-4\right)\prod_{i=1}^{\floor*{\frac{\ell}{2}}-1}c_i^2&\text{, if }\ell\in2\N\\c_{\floor*{\frac{\ell}{2}}}\left(\frac{11}{2}c_{\floor*{\frac{\ell}{2}}}-8\right)\prod_{i=1}^{\floor*{\frac{\ell}{2}}-1}c_i^2&\text{, if }\ell\in2\N+1\end{dcases}\fullstop
\end{align*}
Note that assuming $c_i>c_{i-1}$ leads to $c_{i+1}=\frac{11}{2}-\frac{4}{c_i}<\frac{11}{2}-\frac{4}{c_{i-1}}=c_i$, so that with $c_2=\frac{11}{2}-\frac{8}{7}=\frac{61}{14}>\frac{49}{14}=\frac{7}{2}=c_1$ one obtains iteratively that $c_i\geq\frac{7}{2}$ for all $i\in\{1,\dots,\floor*{\frac{\ell}{2}}\}$ and hence $\det(M_{k_1-3}(c_1))>0$, so that the above systems admit solutions $\mu_i$ and $\nu_i$.

For the general case define $g$ for $k_i-k_j=1$ to be identity except the $(i,j)$- and $(j,i)$-block the same as the above gauge transformation. This transforms the $(i,j)$- and $(j,i)$-block of $F(\alpha,\beta)_{-1/2}$ to zero. Iterating this process yields the desired result, as additional terms that appear all vanish to high enough order.
\end{proof}
\end{lemma}

\begin{lemma}\label{representation theory gamma}
Let $T$ be a solution to Nahm's equations (\ref{nahms eq}) such that for $z<0$ of small modulus
\begin{align*}
T_i(\lambda_N+z)=\frac{\diag\left(\tau_i^{k_n},\dots,\tau_i^{k_1}\right)}{z}+\left(\begin{matrix}z^{\gamma_{11}}O(1)&\dots&z^{\gamma_{1n}}O(1)\\
\vdots&\ddots&\vdots\\
z^{\gamma_{n1}}O(1)&\dots&z^{\gamma_{nn}}O(1)\end{matrix}\right)
\end{align*}
for $\gamma_{ij}\geq0$. Then $2\gamma_{ij}\geq\abs{k_{N-i}-k_{N-j}}$ are integers.
\begin{proof}
Define $T_i'(\lambda_N+z)=z^{-1}\diag\left(\tau_i^{k_n},\dots,\tau_i^{k_1}\right)$ and set
\begin{align*}
\Delta T_i(\lambda_1+z)=(T_i-T_i')(\lambda_N+z)=\left(\begin{matrix}z^{\gamma_{11}}\sum_{l=0}^\infty c_{il}^{11}z^l&\dots&z^{\gamma_{1n}}\sum_{l=0}^\infty c_{il}^{1n}z^l\\
\vdots&\ddots&\vdots\\
z^{\gamma_{n1}}\sum_{l=0}^\infty c_{il}^{n1}z^l&\dots&z^{\gamma_{nn}}\sum_{l=0}^\infty c_{il}^{nn}z^l
\end{matrix}\right)\fullstop
\end{align*}
Nahm's equations then yield
\begin{align*}
(\Delta T_i)\dot{\hphantom{)}}+[\Delta T_j,\Delta T_k]+[\Delta T_j,T_k']+[T_j',\Delta T_k]=0\comma
\end{align*}
$(ijk)$ cyclic permutations of $(123)$, so that for $i'>j'$ the $(i',j')$-block of the $z^{\gamma_{i'j'}-1}$ coefficient gives
\begin{align*}
-\gamma_{i'j'}\left(\begin{matrix}c_{10}^{i'j'}\\c_{20}^{i'j'}\\c_{30}^{i'j'}\end{matrix}\right)=\left(\begin{matrix}0&D_3^{i'j'}&-D_2^{i'j'}\\-D_3^{i'j'}&0&D_1^{i'j'}\\D_2^{i'j'}&-D_1^{i'j'}&0\end{matrix}\right)\left(\begin{matrix}c_{10}^{i'j'}\\c_{20}^{i'j'}\\c_{30}^{i'j'}\end{matrix}\right)\comma
\end{align*}
where $D_i^{i'j'}:=\tau_i^{k_{j'}}(\cdot)-(\cdot)\tau_i^{k_{i'}}$. Let $S^\ell$ be the standard representation of $\su{2}$ of dimension $\ell+1$ and note that in the identification of $\mat{k_{j'}\times k_{i'}}{\C}$ with $S^{k_{j'}-1}\otimes(S^{k_{i'}-1})^*$, $D^{i'j'}$ corresponds to the tensor product representation. In view of this fact, write
\begin{align*}
-\gamma_{i'j'}\left(\sum_{j=1}^3e_j\otimes c_{j0}^{i'j'}\right)=\left(\sum_{i=1}^3\tau_i^2\otimes D_i^{i'j'}\right)\left(\sum_{j=1}^3e_j\otimes c_{j0}^{i'j'}\right)\comma
\end{align*}
where $(e_1,e_2,e_3)$ denotes the standard basis of $S^2$, and
\begin{align*}
\sum_{i=1}^3\tau_i^2\otimes D_i^{i'j'}=\tfrac{1}{2}\left(C(S^2\otimes S^{k_{j'}-1}\otimes(S^{k_{i'}-1})^*)-\idmat\otimes C(S^{k_{j'}-1}\otimes(S^{k_{i'}-1})^*)-C(S^2)\otimes\idmat\right)
\end{align*}
with $C(\cdot)$ denoting the Casimir operator of the respective representation. Using the decompositions
\begin{align*}
S^{k_{j'}-1}\otimes(S^{k_{i'}-1})^*&\cong\bigoplus_{i=1}^{k_{i'}} S^{k_{j'}-k_{i'}+2(i-1)}\comma\\
S^2\otimes S^{k_{j'}-1}\otimes(S^{k_{i'}-1})^*&\cong\bigoplus_{i=1}^{k_{i'}}\left(S^{k_{j'}-k_{i'}+2i}\oplus S^{k_{j'}-k_{i'}+2(i-1)}\oplus S^{k_{j'}-k_{i'}+2(i-2)}\right)
\end{align*}
and $C(S^\ell)=-\frac{\ell(\ell+2)}{4}$, it follows that all eigenvalues of the operator $\sum_{i=1}^3\tau_i^2\otimes D_i^{i'j'}$ must be of the form
\begin{align*}
%&\frac{1}{2}\left(-\frac{(k_1-k_2+2(i-1))(k_1-k_2+2i)}{4}+\frac{(k_1-k_2+2(j-1))(k_1-k_2+2j)}{4}+\frac{2(2+2)}{4}\right)\\
%&\quad=
(j-i)\frac{k_{j'}-k_{i'}}{2}+\frac{j(j-1)}{2}-\frac{i(i-1)}{2}+1
\end{align*}
for $i\in\{0,\dots,k_{i'}+1\}$ and $j\in\{1,\dots,k_{i'}\}$. Since $-\gamma_{i'j'}$ is such an eigenvalue, $\gamma_{i'j'}$ may be written as
\begin{align*}
\gamma=(i-j)\frac{k_{j'}-k_{i'}}{2}+\frac{i(i-1)}{2}-\frac{j(j-1)}{2}-1\fullstop
\end{align*}
This expression is increasing $i\in\{j+1,\dots,k_{i'}+1\}$ for fixed $j\in\{1,\dots,k_{i'}\}$ with minimal non-negative value $\frac{k_{j'}-k_{i'}}{2}+j-1$, which implies $\gamma_{i'j'}\geq\frac{k_{j'}-k_{i'}}{2}$. An analogous argument for $\gamma_{j'i'}$ concludes the proof.
\end{proof}
\end{lemma}

\begin{lemma}\label{matrix to hurtubise normal form}
Let $i>j$ and let $B$ be a matrix with $(i',i')$-block of the form
\begin{align*}
\left(\begin{matrix}
0&&&C_{(m_{N-i'+1}+1)i'}\\
1&\ddots&&\vdots\\
&\ddots&0&C_{(m_{N-i'}-1)i'}\\
&&1&C_{m_{N-i'}i'}
\end{matrix}\right)\comma
\end{align*}
$(i',j')$-block for $i'<i$ and $j'<i'$, as well as $i'=i$ and $j'<j$, of the form
\begin{align*}
\left(\begin{matrix}
C_{(m_{N-i'+1}+k_{N-j'})j'}&\dots&C_{(m_{N-i'+1}+1)j'}\\
0&\dots&0\\
\vdots&\ddots&\vdots\\
0&\dots&0
\end{matrix}\right)\comma
\end{align*}
$(i',j')$-block for $i'>i$ and $j'<i'$, as well as $i'=i$ and $i>j'\geq j$, of the form
\begin{align*}
\left(\begin{matrix}
0&\dots&0&C_{(m_{N-i'+1}+1)j'}\\
\vdots&\ddots&\vdots&\vdots\\
0&\dots&0&C_{(m_{N-i'+1}+k_{N-j'})j'}\\
0&\dots&0&0\\
\vdots&\ddots&\vdots&\vdots\\
0&\dots&0&0
\end{matrix}\right)\comma
\end{align*}
and $(i',j')$-block for $j'>i'$ of the form
\begin{align*}
\left(\begin{matrix}
0&\dots&0&C_{(m_{N-i'+1}+1)j'}\\
\vdots&\ddots&\vdots&\vdots\\
0&\dots&0&C_{m_{N-i'}j'}
\end{matrix}\right)\fullstop
\end{align*}
Conjugation by
\begin{align*}
g_B=\idmat+\sum_{i'=1}^{k_{N-j}-1}\sum_{j'=1}^{k_{N-j}-i}\nu_{i'}E_{(m_{N-i+1}+j')(m_{N-j+1}+i'+j')}\comma
\end{align*}
where $\nu_{i'}$ is defined recursively by $\nu_{1}=C_{(m_{N-i+1}+k_{N-j})j}$ and
\begin{align*}
\nu_{i'}=C_{(m_{N-i+1}+k_{N-j}+1-i')j}+\sum_{j'=1}^{i'-1}C_{(m_{N-j}+1-j')j}\nu_{i'-j'}\comma
\end{align*}
changes the $(i,j)$-block of $B$ to the form where the only possibly non-vanishing row is the first, and possibly changes the last column of the $(i,j')$-block for $j'>j$ and the first row of the $(i,j')$ block for $j'<j$; everything else remains unchanged.
\begin{proof}
Note that the only block of $g_B$ different from the corresponding block of $\idmat$ is the $(i,j)$-block, which is given by
\begin{align*}
\left(\begin{matrix}
0&\nu_1&\dots&\nu_{k_{N-j}-1}\\
&&\ddots&\vdots\\
&&&\nu_1\\
0&0&\dots&0\\
\vdots&\vdots&\ddots&\vdots\\
0&0&\dots&0
\end{matrix}\right)\fullstop
\end{align*}
Moreover, the negative of that is the $(i,j)$-block of $g_B^{-1}$, which is the only block of $g_B^{-1}$ different from $\idmat$. It follows that $Bg_B^{-1}$ agrees with $B$ in all but the $(i,j)$-block, while the $(i,j)$-block is given by
\begin{align*}
\left(\begin{matrix}
0&\dots&0&C_{(m_{N-i+1}+1)j}\\
\vdots&\ddots&\vdots&\vdots\\
0&\dots&0&C_{(m_{N-i+1}+k_{N-j})j}\\
0&\dots&0&0\\
\vdots&\ddots&\vdots&\vdots\\
0&\dots&0&0
\end{matrix}\right)-
\left(\begin{matrix}
0\\
1&0\\
&\ddots&\ddots\\
&&1&0
\end{matrix}\right)
\left(\begin{matrix}
0&\nu_1&\dots&\nu_{k_{N-j}-1}\\
&&\ddots&\vdots\\
&&&\nu_1\\
0&0&\dots&0\\
\vdots&\vdots&\ddots&\vdots\\
0&0&\dots&0
\end{matrix}\right)\fullstop
\end{align*}
The $(i,j)$-block of $g_BBg_B^{-1}$ is then
\begin{align*}
&\left(\begin{matrix}
0&\dots&0&C_{(m_{N-i+1}+1)j}\\
\vdots&\ddots&\vdots&\vdots\\
0&\dots&0&C_{(m_{N-i+1}+k_{N-j})j}\\
0&\dots&0&0\\
\vdots&\ddots&\vdots&\vdots\\
0&\dots&0&0
\end{matrix}\right)-
\left(\begin{matrix}
0&0&\dots&0\\
0&\nu_1&\dots&\nu_{k_{N-j}-1}\\
&&\ddots&\vdots\\
&&&\nu_1\\
0&0&\dots&0\\
\vdots&\vdots&\ddots&\vdots\\
0&0&\dots&0
\end{matrix}\right)\\
&\quad\phantom{{}={}}+\left(\begin{matrix}
0&\nu_1&\dots&\nu_{k_{N-j}-1}\\
&&\ddots&\vdots\\
&&&\nu_1\\
0&0&\dots&0\\
\vdots&\vdots&\ddots&\vdots\\
0&0&\dots&0
\end{matrix}\right)
\left(\begin{matrix}
0&&&C_{(m_{N-j+1}+1)j}\\
1&\ddots&&\vdots\\
&\ddots&0&C_{(m_{N-j}-1)j}\\
&&1&C_{m_{N-j}j}
\end{matrix}\right)\\
&\quad=\left(\begin{matrix}
\nu_1&\dots&\nu_{k_{N-j}-1}&C_{(m_{N-i+1}+1)j}+\sum_{j'=1}^{k_{N-j}-1}\nu_{j'}C_{(m_{N-j+1}+j')j}\\
0&\dots&0&C_{(m_{N-i+1}+2)j}+\sum_{j'=1}^{k_{N-j}-2}\nu_{j'}C_{(m_{N-j+1}+j'+1)j}-\nu_{k_{N-j}-1}\\
\vdots&\ddots&\vdots&\vdots\\
0&\dots&0&C_{(m_{N-i+1}+k_{N-j})j}-\nu_1\\
0&\dots&0&0\\
\vdots&\ddots&\vdots&\vdots\\
0&\dots&0&0
\end{matrix}\right)\\
&\quad=\left(\begin{matrix}
\nu_1&\dots&\nu_{k_{N-j}-1}&C_{(m_{N-i+1}+1)j}+\sum_{j'=1}^{k_{N-j}-1}\nu_{j'}C_{(m_{N-j+1}+j')j}\\
0&\dots&0&0\\
\vdots&\ddots&\vdots&\vdots\\
0&\dots&0&0
\end{matrix}\right)\comma
\end{align*}
where in the last step it was used that from the recursive definition of the $\nu_{i'}$ one obtains
\begin{align*}
\nu_{k_{N-j}-i'}&=C_{(m_{N-i+1}+i'+1)j}+\sum_{j'=1}^{k_{N-j}-i'-1}C_{(m_{N-j}+1-j')j}\nu_{k_{N-j}-i'-j'}\\
&=C_{(m_{N-i+1}+i'+1)j}+\sum_{j'=1}^{k_{N-j}-i'-1}\nu_{j'}C_{(m_{N-j+1}+1+j'-i')j}\fullstop
\end{align*}
Finally, observing that the $(i,j')$-block of $g_BBg_B^{-1}$ for $j'\neq j$ differs from the one of $Bg_B^{-1}$ only by
\begin{align*}
\left(\begin{matrix}
0&\nu_1&\dots&\nu_{k_{N-j}-1}\\
&&\ddots&\vdots\\
&&&\nu_1\\
0&0&\dots&0\\
\vdots&\vdots&\ddots&\vdots\\
0&0&\dots&0
\end{matrix}\right)
\left(\begin{matrix}
C_{(m_{N-i+1}k_{N-j'})j'}&\dots&C_{(m_{N-i+1}+1)j'}\\
0&\dots&0\\
\vdots&\ddots&\vdots\\
0&\dots&0
\end{matrix}\right)=0\comma
\end{align*}
if $j'<j$, by
\begin{align*}
\left(\begin{matrix}
0&\nu_1&\dots&\nu_{k_{N-j}-1}\\
&&\ddots&\vdots\\
&&&\nu_1\\
0&0&\dots&0\\
\vdots&\vdots&\ddots&\vdots\\
0&0&\dots&0
\end{matrix}\right)
\left(\begin{matrix}
0&\dots&0&C_{(m_{N-i+1}+1)j'}\\
\vdots&\ddots&\vdots&\vdots\\
0&\dots&0&C_{(m_{N-i+1}+k_{N-j'})j'}\\
0&\dots&0&0\\
\vdots&\ddots&\vdots&\vdots\\
0&\dots&0&0
\end{matrix}\right)
\end{align*}
if $j<j'<i$, and by
\begin{align*}
\left(\begin{matrix}
0&\nu_1&\dots&\nu_{k_{N-j}-1}\\
&&\ddots&\vdots\\
&&&\nu_1\\
0&0&\dots&0\\
\vdots&\vdots&\ddots&\vdots\\
0&0&\dots&0
\end{matrix}\right)
\left(\begin{matrix}
0&\dots&0&C_{(m_{N-i+1}+1)j'}\\
\vdots&\ddots&\vdots&\vdots\\
0&\dots&0&C_{m_{N-i}j'}
\end{matrix}\right)
\end{align*}
if $j'>i$, then concludes the proof.
\end{proof}
\end{lemma}

\begin{lemma}\label{matrix to hurtubise normal form surjectivity}
Let $i>j$ and let $B'$ be a matrix with $(i',i')$-block of the form
\begin{align*}
\left(\begin{matrix}
0&&&C_{(m_{N-i'+1}+1)i'}'\\
1&\ddots&&\vdots\\
&\ddots&0&C_{(m_{N-i'}-1)i'}'\\
&&1&C_{m_{N-i'}i'}'
\end{matrix}\right)\comma
\end{align*}
$(i',j')$-block for $i'<i$ and $j'<i'$, as well as $i'=i$ and $j'\leq j$, of the form
\begin{align*}
\left(\begin{matrix}
C_{(m_{N-i'+1}+k_{N-j'})j'}'&\dots&C_{(m_{N-i'+1}+1)j'}'\\
0&\dots&0\\
\vdots&\ddots&\vdots\\
0&\dots&0
\end{matrix}\right)\comma
\end{align*}
$(i',j')$-block for $i'>i$ and $j'<i'$, as well as $i'=i$ and $i>j'>j$, of the form
\begin{align*}
\left(\begin{matrix}
0&\dots&0&C_{(m_{N-i'+1}+1)j'}'\\
\vdots&\ddots&\vdots&\vdots\\
0&\dots&0&C_{(m_{N-i'+1}+k_{N-j'})j'}'\\
0&\dots&0&0\\
\vdots&\ddots&\vdots&\vdots\\
0&\dots&0&0
\end{matrix}\right)\comma
\end{align*}
and $(i',j')$-block for $j'>i'$ of the form
\begin{align*}
\left(\begin{matrix}
0&\dots&0&C_{(m_{N-i'+1}+1)j'}'\\
\vdots&\ddots&\vdots&\vdots\\
0&\dots&0&C_{m_{N-i'}j'}'
\end{matrix}\right)\fullstop
\end{align*}
Then $g_B$ conjugates $B$ to $B'$, where $B=g^{-1}B'g$ with $g$ having only possibly the $(i,j)$-block different from $\idmat$, namely given by
\begin{align*}
\left(\begin{matrix}
0&C_{(m_{N-i+1}+k_{N-j})j}'&\dots&C_{(m_{N-i+1}+2)j}'\\
&&\ddots&\vdots\\
&&&C_{(m_{N-i+1}+k_{N-j})j}'\\
0&0&\dots&0\\
\vdots&\vdots&\ddots&\vdots\\
0&0&\dots&0
\end{matrix}\right)\comma
\end{align*}
is of the form as in lemma \ref{matrix to hurtubise normal form}.
\begin{proof}
Define $\nu_{i'}=C_{(m_{N-i+1}+k_{N-j}+1-i'}$. As above, $B'g_B$ agrees with $B'$ in all but the $(i,j)$-block, and the $(i,j)$-block of $B'g_B$ is given by
\begin{align*}
\left(\begin{matrix}
C_{(m_{N-i+1}+k_{N-j})j}'&\dots&C_{(m_{N-i+1}+1)j}'\\
0&\dots&0\\
\vdots&\ddots&\vdots\\
0&\dots&0
\end{matrix}\right)+
\left(\begin{matrix}
0\\
1&0\\
&\ddots&\ddots\\
&&1&0
\end{matrix}\right)
\left(\begin{matrix}
0&\nu_1&\dots&\nu_{k_{N-j}-1}\\
&&\ddots&\vdots\\
&&&\nu_1\\
0&0&\dots&0\\
\vdots&\vdots&\ddots&\vdots\\
0&0&\dots&0
\end{matrix}\right)\fullstop
\end{align*}
The $(i,j)$-block of $g^{-1}_BB'g_B$ is then
\begin{align*}
&\left(\begin{matrix}
C_{(m_{N-i+1}+k_{N-j})j}'&\dots&C_{(m_{N-i+1}+1)j}'\\
0&\dots&0\\
\vdots&\ddots&\vdots\\
0&\dots&0
\end{matrix}\right)+
\left(\begin{matrix}
0&0&\dots&0\\
0&\nu_1&\dots&\nu_{k_{N-j}-1}\\
&&\ddots&\vdots\\
&&&\nu_1\\
0&0&\dots&0\\
\vdots&\vdots&\ddots&\vdots\\
0&0&\dots&0
\end{matrix}\right)\\
&\quad\phantom{{}={}}-\left(\begin{matrix}
0&\nu_1&\dots&\nu_{k_{N-j}-1}\\
&&\ddots&\vdots\\
&&&\nu_1\\
0&0&\dots&0\\
\vdots&\vdots&\ddots&\vdots\\
0&0&\dots&0
\end{matrix}\right)
\left(\begin{matrix}
0&&&C_{(m_{N-j+1}+1)j}'\\
1&\ddots&&\vdots\\
&\ddots&0&C_{(m_{N-j}-1)j}'\\
&&1&C_{m_{N-j}j}'
\end{matrix}\right)\\
&\quad=\left(\begin{matrix}
0&\dots&0&C_{(m_{N-i+1}+1)j}'-\sum_{j'=1}^{k_{N-j}-1}\nu_{j'}C_{(m_{N-j+1}+j')j}'\\
0&\dots&0&-\sum_{j'=1}^{k_{N-j}-2}\nu_{j'}C_{(m_{N-j+1}+j'+1)j}'+\nu_{k_{N-j}-1}\\
\vdots&\ddots&\vdots&\vdots\\
0&\dots&0&\nu_1\\
0&\dots&0&0\\
\vdots&\ddots&\vdots&\vdots\\
0&\dots&0&0
\end{matrix}\right)\fullstop
\end{align*}
The proof is concluded in analogy to the above proof.
\end{proof}
\end{lemma}

\begin{corollary}\label{bijection of matrices for normal forms}
Let $B=B_{21}$ be a matrix as in lemma \ref{matrix to hurtubise normal form} for $(i,j)=(2,1)$ and define $B_{i(j+1)}=g_{B_{ij}}B_{ij}g_{B_{ij}}^{-1}$ for $j+1<i\leq m$ and $B_{(i+1)1}=g_{B_{i(i-1)}}B_{i(i-1)}g_{B_{i(i-1)}}^{-1}$ for $i<m$. Then the mapping rule $B\mapsto gBg^{-1}$ with
\begin{align*}
g=\prod_{i=1}^{n-1}\prod_{j=1}^{i-1}g_{B_{(N-i)(i-j)}}
\end{align*}
induces a bijection between the set of matrices $B$ as described in lemma \ref{matrix to hurtubise normal form} for $(i,j)=(2,1)$ and the set of matrices with $(i,j)$-block for $j\geq i$ of the same form as $B$, and $(i,j)$-block for $j<i$ with all but possibly the first row vanishing.
\begin{proof}
By lemma \ref{matrix to hurtubise normal form surjectivity}, the given map is surjective and the injectivity follows by using that $g$ can be reconstructed from an element of the image.
\end{proof}
\end{corollary}

\section*{Acknowledgement}

The work presented here forms part of the Author's PhD under the supervision of Prof. Dr. Roger Bielawski to whom the Author is very grateful for suggesting this problem and for the many inspiring discussions.

\end{appendices}

% --- Bibliography ---

\printbibliography

% --- Document ends here ---

\end{document}